\documentclass[review,hidelinks,onefignum,onetabnum]{siamart220329}

\usepackage{lipsum}
\usepackage{amsfonts}
\usepackage{graphicx}
\usepackage{epstopdf}
\ifpdf
  \DeclareGraphicsExtensions{.eps,.pdf,.png,.jpg}
\else
  \DeclareGraphicsExtensions{.eps}
\fi

\newsiamremark{remark}{Remark}
\newsiamremark{hypothesis}{Hypothesis}
\crefname{hypothesis}{Hypothesis}{Hypotheses}
\newsiamthm{claim}{Claim}

\headers{The Optimal Production Transport: Model and Algorithm}{Jie Fan, Tianhao Wu, and Hao Wu}

\title{The Optimal Production Transport: Model and Algorithm\thanks{Submitted to the editors DATE.
\funding{National Natural Science Foundation of China (Grant No. 12271289). }}}

\author{Jie Fan\thanks{Department of Mathematical Sciences, Tsinghua University, Beijing, P.R. China 
(\email{fanj21@mails.tsinghua.edu.cn}, 
\email{wuth20@mails.tsinghua.edu.cn}.}
\and Tianhao Wu\footnotemark[2]
\and Hao Wu\thanks{Corresponding author. Department of Mathematical Sciences, Tsinghua University, Beijing, P.R. China (\email{hwu@tsinghua.edu.cn}).}
}

\usepackage{amsopn}

\ifpdf
\hypersetup{
  pdftitle={The Optimal Production Transport: Model and Algorithm},
  pdfauthor={Jie Fan, Tianhao Wu, and Hao Wu}
}
\fi

\usepackage{bm}
\usepackage{multirow}
\usepackage{subfigure}
\usepackage{caption}
\usepackage[noend]{algorithmic}
\usepackage{mathtools}
\usepackage{setspace}
\usepackage{amssymb}

\begin{document}
\nolinenumbers

\maketitle

\begin{abstract}
In this paper, we propose the optimal production transport model, which is an extension of the classical optimal transport model. 
We observe in economics, the production of the factories can always be adjusted within a certain range, while the classical optimal transport does not take this situation into account. Therefore, differing from the classical optimal transport, one of the marginals is allowed to vary within a certain range in our proposed model. 
To address this, we introduce a multiple relaxation optimal production transport model and propose the generalized alternating Sinkhorn algorithms, inspired by the Sinkhorn algorithm and the double regularization method. By incorporating multiple relaxation variables and multiple regularization terms, the inequality and capacity constraints in the optimal production transport model are naturally satisfied. Alternating iteration algorithms are derived based on the duality of the regularized model.
We also provide a theoretical analysis to guarantee the convergence of our proposed algorithms. Numerical results indicate significant advantages in terms of accuracy and efficiency.
Furthermore, we apply the optimal production transport model to the coal production and transport problem. Numerical simulation demonstrates that our proposed model can save the production and transport cost by 13.17\%. 

\end{abstract}



\section{Introduction}
\label{sec:introduction}
 
The widely studied Optimal Transport~(OT) theory can be traced back to 1781 in France, where Monge first formalized it within the field of civil engineering~\cite{monge1781}. 
Major advances in the OT theory were made by Soviet mathematician and economist Leonid Kantorovich \cite{kantorovich1942,villani2003topics}. 
He introduced a relaxation technique to transform the OT problem into a linear programming problem and further provided an economic interpretation \cite{OTinEconomic}, that is the optimal allocation and utilization of resources in the whole society. 
In 1975, the Nobel Memorial Prize in Economic Sciences, which he shared with Tjalling Koopmans, was given ``for their contributions to the theory of optimum allocation of resources.'' 
Subsequently, OT has been widely generalized and successfully applied in economics, 
including the classical discrete choice model \cite{chiong2016duality,galichon2022cupid},
the partial identification with random sets \cite{galichon2021inference,galichon2011set}, 
the hedonic equilibrium problem \cite{ekeland2004identification,matzkin2003nonparametric},  and the price discrimination and implementability \cite{laffont1988fundamentals}, among others.

In the classical OT problem, researchers consider how to transport goods between fixed production and consumption, in other words, to obtain the optimal transport plan. 
However, in economics, the production between different factories always can be adjusted within a certain range, and various factories may have different costs to produce one unit of goods \cite{kasilingam1998logistics,stock2001strategic,mangan2016global}. 
This situation requires us to find the optimal production transport plan simultaneously with the minimum total cost, i.e., the summation of transport cost and production cost. To the best of our knowledge, this type of extension has not been considered in existing OT models. 
Therefore, it is essential to explore the formulation of this variation. 

In this paper, we generalize the classical OT model to the Optimal Production Transport~(OPT) model, which allows one of the marginal to vary within a certain range. 
Specifically, the OPT model can be expressed as 
\begin{equation}
\label{OPTC}
\begin{aligned}
 		\min_{\gamma, u} & \quad \int_{\Omega_1 \times \Omega_2} C(x,y)\gamma(x,y)dx dy +\int_{\Omega_1}p(x) u(x) dx,\\
\mathrm{s.t.} &  \quad \int_{\Omega_2}\gamma(x,y) d y=u(x), \ \int_{\Omega_1}\gamma(x,y) d x=v(y), \\
		& \quad  \theta(x,y) \leq \gamma(x,y)\leq \eta(x,y), \ \hat{u}(x)\leq u(x) \leq \overline{u}(x),  
\end{aligned}
\end{equation}
where $C$ denotes the transport cost, and $p$ denotes the production cost. $\gamma$ represents the transport plan with the lower bound $\theta$ and the upper bound $\eta$, and $u$ represents the production plan with the lower bound $\hat{u}$ and the upper bound $\overline{u}$. 
Note that when $\theta=0$, $\eta=+\infty$, and $\hat{u}=\overline{u}$, this model degenerates into the classic OT model.

Although there exist a large number of algorithms designed for the classical OT model \cite{peyre2017computational, IBP}, the OPT model brings out new difficulties, which hinder applying these algorithms to the OPT model directly. More specifically, the primary difficulties arise in two aspects: (1) The objective function contains an additional term of a variable marginal compared with the classical OT, and the duality is in a max-min form when directly applying alternating algorithms for the classical OT. (2) The upper bound and lower bound of marginal $u$ introduce inequality constraints and complementary slackness conditions in the duality, which leads to the coupling of dual and primal variables.

To address difficulty (1) in the OPT model, we first rigorously prove that it can be equivalently transformed into a new model. In this model, the production cost can be absorbed into the transport cost such that the objective function is the same as the classical OT. 
For difficulty (2), we introduce the Multiple Relaxation Optimal Production Transport (MR-OPT) model with multiple relaxation variables, which is equivalent to the OPT model. 
We then incorporate multiple regularization terms into the MR-OPT model, inspired by the Double Regularization Method (DRM) \cite{wu2022double} for the Capacity-constrained Optimal Transport~(COT) problem.
From this, we design two types of Generalized Alternating Sinkhorn algorithms~(GAS-I and GAS-II) to solve the regularized MR-OPT model.\footnote{The Iterative Bregman Projections~(IBP) algorithm \cite{IBP} based on the KL divergence is also developed as a baseline for the OPT model, see Appendix \ref{sec:IBP} for related discussion.}

More specifically, in the MR-OPT model, to overcome the difficulty arising from the inequality constraints on the marginal, we introduce additional multiple slack variables to convert inequality constraints into equality constraints. 
By incorporating multiple regularization terms, we define the regularized MR-OPT model, which naturally satisfies the non-negative constraints of the slack variables. 
A critical observation is that solving the regularized MR-OPT model is equivalent to iteratively solving three sets of auxiliary variables in duality, and further, one set of auxiliary variables can be eliminated according to their correlation. 
During each iteration, it is noteworthy that solving each variable corresponds to finding the unique root of a one-dimensional equation, and can be efficiently solved by Newton's method. 
This process constitutes the Generalized Alternating Sinkhorn-I~(GAS-I) algorithm. 
Additionally, by noting the conserved total mass within the MR-OPT model and taking this condition into account, we present the Generalized Alternating Sinkhorn-II (GAS-II) algorithm.   
The convergence and the number of iterations of GAS-II are also theoretically guaranteed.

To illustrate the effectiveness and efficiency of the model and algorithm, we conduct extensive numerical experiments and show the advantages of the GAS algorithms in accuracy, efficiency, and memory consumption. 
For practical applications, we consider the situation of coal production and consumption in several provinces in the Chinese mainland, which is an important and useful problem mentioned in plenty of studies \cite{kantorovich1960mathematical,cao2006coal,kozan2012demand}. 
Based on these existing studies, we formulate the coal production and transport problem as an OPT model and solve it with the GAS algorithms. 
The results demonstrate that our OPT model may provide useful suggestions for national coal production and transport.

The rest of the paper is organized as follows, Section \ref{sec:MRM} introduces and simplifies the discrete optimal production transport model, then presents the multiple relaxation optimal production transport model with multiple regularization terms. 
The two kinds of Generalized Alternating Sinkhorn algorithms~(GAS-I and GAS-II), along with the convergence guarantee of GAS-II, are discussed in Section \ref{MRM-alg}. 
In Section \ref{sec:numerical}, numerical experiments demonstrate the advantages of our proposed algorithms in both accuracy and efficiency. To illustrate the applications of our model, the coal production and transport problem is formulated to the OPT model and solved with the GAS-I algorithm in Section \ref{sec:applicationA}. Finally, Section \ref{sec:conclusions} concludes this paper.

\section{Multiple Relaxation OPT Model}
\label{sec:MRM}
 
Consider $\bm{u} = \left(u_1, \cdots,  u_{N_1}\right), \ \bm{v}=\left( v_1, \cdots, v_{N_2}\right)$
are two discrete probabilistic distributions. $C_{ij}$ is the unit cost of transporting, 
and $p_{i}$ is the unit cost of producing. 
The continuous OPT model~\eqref{OPTC} can be discretized to the following form
\begin{equation}
\label{OPT_d0}
    \begin{gathered}
        \min_{\bm{\gamma}, \bm{u}} \quad \sum_{i=1}^{N_1}\sum_{j=1}^{N_2} C_{ij}\gamma_{ij}+\sum_{i=1}^{N_1}p_i u_i\\
        \mathrm{s.t.} \quad \sum_{j=1}^{N_2} \gamma_{ij} = u_i,  \ \sum_{i=1}^{N_1} \gamma_{ij} = v_j, \ 
        \theta_{ij} \leq \gamma_{ij}\leq \eta_{ij}, \  \hat{u}_i\leq u_i \leq \overline{u}_i, 
    \end{gathered}
\end{equation}
Here, the transport plan $\gamma_{ij}$ represents the mass transported from the $i$-th source to $j$-th target, with the upper bound $\eta_{ij}$ and the lower bound $\theta_{ij}$. 
The production plan $u_{i}$ represents the total mass produced by the $i$-th source, with the upper bound $\overline{u}_i$ and the lower bound $\hat{u}_{i}$. 
The first two equality constraints guarantee that the mass produced by each source are entirely shipped out, and the needs of each target are met exactly. 
In the OPT model, both the transport plan $\gamma_{ij}$ and the production plan $u_i$ are variables.
For convenience, assume that the feasible region of problem \cref{OPT_d0} is non-empty.
It is also worth noting that our discussion is general for any $N_1$ and $N_2$, and thus we set $N_1 = N_2 = N$ in the rest of the paper for the sake of simplicity.

The discrete OPT model \cref{OPT_d0} can be equivalently converted to the following simplified OPT model
\begin{equation}\label{OPT_obj}
\begin{gathered}
    \min_{\bm{\gamma},\bm{u}} \quad \sum_{i=1}^{N}\sum_{j=1}^{N} \widetilde{C}_{ij}(\gamma_{ij}-\theta_{ij}), \\
\begin{aligned}
    \mathrm{s.t.} & \quad \sum_{j=1}^{N} (\gamma_{ij}-\theta_{ij})=u_i-\sum_{j=1}^{N} \theta_{ij}, \ \sum_{i=1}^{N} (\gamma_{ij}-\theta_{ij})=v_j-\sum_{i=1}^{N} \theta_{ij}, \\ 
    & \quad 0\leq \gamma_{ij}-\theta_{ij}\leq \eta_{ij}-\theta_{ij}, \ \hat{u}_i-\sum_{j=1}^{N} \theta_{ij} \leq u_i-\sum_{j=1}^{N} \theta_{ij} \leq \overline{u}_i-\sum_{j=1}^{N} \theta_{ij}. 
\end{aligned}
\end{gathered}
\end{equation}
where $\widetilde{C}_{ij}= C_{ij}+p_i$.
The equivalence between \cref{OPT_d0} and \cref{OPT_obj} is ensured by the following equation
\begin{multline*}
    \sum_{i=1}^{N}\sum_{j=1}^{N} C_{ij}\gamma_{ij}+\sum_{i=1}^{N}p_i u_i
    = \sum_{i=1}^{N}\sum_{j=1}^{N} C_{ij}\gamma_{ij}+\sum_{i=1}^{N}p_i \sum_{j=1}^{N}\gamma_{ij} \\
    = \sum_{i=1}^{N}\sum_{j=1}^{N} \left( C_{ij}+p_i \right) \gamma_{ij}
    = \sum_{i=1}^{N}\sum_{j=1}^{N} \widetilde{C}_{ij} \gamma_{ij}, 
\end{multline*}
The above transformation absorbs the additional term of the variable marginal in the objective function, resulting in the same objective function as the classical OT.

Then, the discrete OPT model~\cref{OPT_obj} can be reformulated as
\begin{equation}\label{eq:theta}
\begin{gathered}
    \min_{\bm{\gamma},\bm{u}} \quad \sum_{i=1}^{N}\sum_{j=1}^{N} \widetilde{C}_{ij}\widetilde{\gamma}_{ij}, \\
\begin{aligned}
    \mathrm{s.t.} & \quad \sum_{j=1}^{N} \widetilde{\gamma}_{ij}=\widetilde{u}_i, \ \sum_{i=1}^{N} \widetilde{\gamma}_{ij}=\widetilde{v}_j, \
 0\leq \widetilde{\gamma}_{ij}\leq \widetilde{\eta}_{ij}, \ \widetilde{\hat{u}}_i \leq \widetilde{u}_i \leq \widetilde{\overline{u}}_i, \\
 &\quad  \widetilde{\gamma}_{ij}=\gamma_{ij}-\theta_{ij},\ \widetilde{u}_{i}=u_i-\sum_{j=1}^{N} \theta_{ij}, \ \widetilde{v}_{j}=v_j-\sum_{i=1}^{N} \theta_{ij},\\
 &\quad  \widetilde{\eta}_{ij}=\eta_{ij}-\theta_{ij},\
 \widetilde{\hat{u}}_{i}=\hat{u}_i-\sum_{j=1}^{N} \theta_{ij}, \ \widetilde{\overline{u}}_{i}=\overline{u}_i-\sum_{j=1}^{N} \theta_{ij}.
\end{aligned}
\end{gathered}
\end{equation}

For convenience, `` $\widetilde{\;\;}$ " is dropped in the following context.
Based on the above analysis, we consider the following simplified OPT model 
 \begin{equation}\label{OPT_d}
     \begin{gathered}
     \min_{\bm{\gamma,u}} \quad \sum_{i=1}^{N}\sum_{j=1}^{N} C_{ij} \gamma_{ij} \\
     \mathrm{s.t.} \quad \sum_{j=1}^{N} \gamma_{ij} = u_i,  \ \sum_{i=1}^{N} \gamma_{ij} = v_j, \ 
        0 \leq \gamma_{ij} \leq \eta_{ij}, \ \hat{u}_i \leq u_i \leq \overline{u}_i.  
     \end{gathered}
 \end{equation}

Compared to the classical OT, the main difference in \cref{OPT_d} is the additional inequality constraints on variables $\bm{u}$, which are coupled with variables $\bm{\gamma}$. 
When directly applying the Sinkhorn algorithm for this model, extra variables $\bm{u}$ can not be updated since only variables concerning transport plan $\bm{\gamma}$ are involved in the alternating scheme.
Thus, we consider the introduction of the relaxation variables $\bm{z},\bm{w}\in \mathbb{R}_{+}^{N}$ into \cref{OPT_d}, where $\mathbb{R_+}$ is the set of positive real numbers. That is
\begin{equation}\label{eq-relax}
\begin{gathered}
        \min_{\bm{\gamma}, \bm{z},\bm{w}} \quad \sum_{i=1}^{N}\sum_{j=1}^{N} C_{ij} \gamma_{ij} \\
        \begin{aligned}
        \text{s.t.} & \quad \sum_{j=1}^{N} \gamma_{ij} + z_i = \overline{u}_i, \ \sum_{j=1}^{N} \gamma_{ij} - w_i = \hat{u}_i, \\
        & \quad  \sum_{i=1}^N \gamma_{ij} = v_j, \
        0 \leq \gamma_{ij} \leq \eta_{ij}, \ z_i \geq 0, \ w_i \geq 0. 
        \end{aligned}
\end{gathered}
\end{equation}
There are multiple relaxation variables in \cref{eq-relax},
yet it is still equivalent to the OPT model \cref{OPT_d}.
Thus we call it the Multiple Relaxation Optimal Production Transport~(MR-OPT) model.
Note that the MR-OPT model only has equality constraints with additional non-negative variables $\bm{z}$ and $\bm{w}$, which enables us to design efficient numerical algorithms that update these variables alternately \cite{nocedal1999numerical}.

To ensure that the inequality constraints in \cref{eq-relax} hold naturally, we introduce multiple regularization terms into the original model, and the regularized model is
\begin{equation}\label{MRM}
\begin{aligned}
& \mathtoolsset{multlined-width=0.9\displaywidth}
\begin{multlined}
    \min_{\bm{\gamma}, \bm{z}, \bm{w}} \quad \sum_{i=1}^{N}\sum_{j=1}^{N} \left(C_{ij}\gamma_{ij}+\varepsilon \gamma_{ij} \ln(\gamma_{ij})+\varepsilon (\eta_{ij}-\gamma_{ij}) \ln(\eta_{ij}-\gamma_{ij})\right) \\
    +\sum_{i=1}^{N} \varepsilon z_i (\ln z_i-1)+\sum_{i=1}^{N} \varepsilon w_i (\ln w_i-1)
\end{multlined}
\\
& \text{   s.t.} \quad \sum_{j=1}^{N} \gamma_{ij} + z_i = \overline{u}_i, \ \sum_{j=1}^{N} \gamma_{ij} - w_i = \hat{u}_i, \ \sum_{i=1}^N \gamma_{ij} = v_j,
\end{aligned}
\end{equation}
where $\varepsilon$ is the regularization parameter. 
This is inspired by the double regularized problem proposed in \cite{wu2022double}.
Note that the regularized MR-OPT model is strictly convex, and the feasible set is only formed by several equality constraints. This facilitates the design of numerical algorithms in which each variable is alternately updated by optimizing a convex problem. 
Despite multiple regularization terms in the regularized MR-OPT model \cref{MRM}, we notice that the convergence of \cref{MRM} w.r.t. $\varepsilon$ still can be proved. That is, the regularized MR-OPT model converges to the original model \cref{OPT_d} as $\varepsilon$ goes to $0$. The theorem and its proof are as follows.

\begin{theorem}
	The optimal solution of the regularized MR-OPT model \cref{MRM} is unique. Let ($\bm{\gamma}_{\varepsilon}, \bm{z}_{\varepsilon}, \bm{w}_{\varepsilon}$) be the optimal solution of model \cref{MRM}. There exists one of the optimal solutions of the OPT model \cref{OPT_d} denoted by ($\bm{\gamma}^*,\bm{ u}^*$) such that $\bm{\gamma}_{\varepsilon}$ converges to $\bm{\gamma}^*$, $\bm{z}_{\varepsilon}$ converges to $\bm{\overline{u}}-\bm{u}^*$ and $\bm{w}_{\varepsilon}$ converges to $\bm{u}^*-\bm{\hat{u}}$ as $\varepsilon\rightarrow 0$. 
\end{theorem}
\begin{proof}
	Firstly, we prove that the optimal solution of model \cref{MRM} is unique. As we all know, $h_1(x)=x\ln x+(1-x)\ln (1-x)$ and  $h_2(x)=x(\ln x -1)$ are strictly convex functions, so the regularized MR-OPT model \cref{MRM} is still strictly convex and has a unique optimal solution. 

    Then we prove the convergence w.r.t. $\varepsilon$. With the fact
    \begin{equation*}
        \bm{\overline{u}}-\bm{u}^*=\bm{\overline{u}}-\bm{\gamma}^* \bm{1}_N, \ \bm{\gamma}_{\varepsilon} \bm{1}_N + \bm{z}_{\varepsilon} = \bm{\overline{u}}, 
    \end{equation*}
    where $\bm{1}_N$ denotes an $N$-size vector in which each element is $1$, $\bm{z}_{\varepsilon}$ converges to $\bm{\overline{u}}-\bm{u}^*$ if and only if $\bm{\gamma}_{\varepsilon}$ converges to $\bm{\gamma}^*$. 
    Similarly, $\bm{w}_{\varepsilon}$ converges to $\bm{u}^*-\bm{\hat{u}}$ if and only if $\bm{\gamma}_{\varepsilon}$ converges to $\bm{\gamma}^*$, so we only prove that $\bm{\gamma}_{\varepsilon}$ converges to $\bm{\gamma}^*$ in the following. 
     
 For any sequence $\{\varepsilon_k\}_{k=1}^{\infty}$, where $\varepsilon_k>0$ and $\lim\limits_{k\rightarrow \infty} \varepsilon_k=0$. Suppose the optimal solution of model \cref{MRM} with the regularization parameter $\varepsilon_k$ is $\bm{\gamma}_k$. 
 The feasible set of the MR-OPT model \cref{MRM} can be equivalently converted to 
 $$\Gamma' = \left\{\bm{\gamma}\in \mathbb{R}^{N\times N} \mid \bm{\hat{u}}\leq \bm{\gamma}\bm{1}_N \leq  \bm{\overline{u}}, \bm{\gamma}^{T} \bm{1}_N=\bm{v}, \bm{0} \leq \bm{\gamma} \leq \bm{\eta} \right\}, $$
 where $\bm{A} \leq \bm{B}$ denotes that every element in the matrix (or vector) $\bm{A}$ is not larger than the corresponding element in the matrix (or vector) $\bm{B}$. 
 Consider that $\Gamma'$ is closed and bounded, there exists a subsequence of $\bm{\gamma}_k$ converging to $\hat{\bm{\gamma}}\in \Gamma'$. For the sake of simplicity, we still use the same symbol $\bm{\gamma}_k$ to represent the subsequence.
 
 Thus, $\bm{\gamma}_{k}, \bm{\gamma}^*$, and $\hat{\bm{\gamma}}$ are feasible solutions of the OPT model \cref{OPT_d} and the MR-OPT model \cref{MRM}. Because $\bm{\gamma}^*$ is the optimal solution of  \cref{OPT_d}, for any $\bm{\gamma}_k$, $\langle \bm{C}, \bm{\gamma}^*\rangle \leq \langle \bm{C}, \bm{\gamma}_k \rangle$ holds. Here, $\langle \bm{A}, \bm{B} \rangle = \sum_i \sum_j a_{ij} b_{ij}$. 
 Because $\bm{\gamma}_k$ is the optimal solution of  \cref{MRM} with $\varepsilon=\varepsilon_k$, 
 \begin{multline*}
     \langle \bm{C}, \bm{\gamma}_k\rangle \! +\! \varepsilon_k(\langle \bm{\gamma}_k, \ln \bm{\gamma}_k \rangle\! +\!  \langle \bm{\eta}-\bm{\gamma}_k, \ln (\bm{\eta}-\bm{\gamma}_k) \rangle  
    \! +\! \langle \bm{z}_k, \ln \bm{z}_k -\bm{1}_N \rangle\! +\!  \langle \bm{w}_k, \ln \bm{w}_k -\bm{1}_N \rangle) \\
     \leq \langle \bm{C}, \bm{\gamma}^* \rangle\!  +\! \varepsilon_k(\langle \bm{\gamma}^*, \ln \bm{\gamma}^* \rangle\! +\! \langle \bm{\eta}-\bm{\gamma}^*, \ln (\bm{\eta}-\bm{\gamma}^*) \rangle  
     \! +\! \langle \bm{z}^*, \ln \bm{z}^* -\bm{1}_N \rangle\! +\! \langle \bm{w}^*, \ln \bm{w}^* -\bm{1}_N \rangle)
 \end{multline*}
 holds, where $\bm{z}^*=\bm{\overline{u}}-\bm{\gamma}^*\bm{1}_N, \bm{w}^*=\bm{\gamma}^* \bm{1}_N-\bm{\hat{u}}$.
 So we have
 \begin{equation*}
     \begin{aligned}
      0 & \leq \langle \bm{C}, \bm{\gamma}_k\rangle - \langle \bm{C}, \bm{\gamma}^* \rangle\\
     &
     \mathtoolsset{multlined-width=0.95\displaywidth}
     \begin{multlined}
         \leq \varepsilon_k \left(\langle \bm{\gamma}^*, \ln \bm{\gamma}^* \rangle+ \langle \bm{\eta}-\bm{\gamma}^*, \ln (\bm{\eta}-\bm{\gamma}^*) \rangle  + \langle \bm{z}^*, \ln \bm{z}^* -\bm{1}_N \rangle+ \langle \bm{w}^*, \ln \bm{w}^* -\bm{1}_N\rangle \right. \\
     \left. - \langle \bm{\gamma}_k, \ln \bm{\gamma}_k \rangle- \langle \bm{\eta}-\bm{\gamma}_k, \ln (\bm{\eta}-\bm{\gamma}_k) \rangle  -\langle \bm{z}_k, \ln \bm{z}_k -\bm{1}_N \rangle- \langle \bm{w}_k, \ln \bm{w}_k -\bm{1}_N \rangle  \right).
     \end{multlined}
     \end{aligned}
 \end{equation*}
 Note that the inner product function and logarithmic function are continuous and the definition domain is closed and bounded, so each term in the right-hand equation is bounded with different $k$. 
 Now let the limit $k\rightarrow \infty$, then $\varepsilon_k\rightarrow 0$ and consequently $\langle \bm{C}, \bm{\gamma}_k\rangle - \langle \bm{C}, \bm{\gamma}^* \rangle \leq 0$ holds, which means that $\langle \bm{C}, \bm{\gamma}_k\rangle - \langle \bm{C}, \bm{\gamma}^* \rangle = 0$. Hence, $\hat{\bm{\gamma}}$ is the optimal solution of the OPT model \cref{OPT_d}, and the theorem holds. 
\end{proof}

\section{Generalized Alternating Sinkhorn Algorithm}\label{MRM-alg}
To solve the regularized MR-OPT model \cref{MRM}, we analyze the Lagrangian dual and design the Generalized Alternating Sinkhorn-I~(GAS-I) algorithm updating the dual variables alternately, which is described in detail in \cref{subsec:GAS-I}.
We then note a conserved total mass within the MR-OPT model \cref{eq-relax}, specified by the condition $\bm{1}_N^T\bm{\gamma}\bm{1}_N=\bm{1}_N^T\bm{v}$. 
Taking this condition into account, the regularized MR-OPT model is extended with two additional constraints on the slack variables, and the Generalized Alternating Sinkhorn-II~(GAS-II) algorithm is presented in \cref{subsec:GAS-II} following this extension.
This extension ensures that the total mass remains unchanged during the alternating iterations, which makes it possible to analyze the upper and lower bounds of the dual variables in each iteration. 
Consequently, a theoretical analysis regarding the convergence of the GAS-II algorithm and the number of iterations required for convergence is provided in \cref{subsec:convergence}.

\subsection{GAS-I Algorithm}\label{subsec:GAS-I}
By introducing Lagrange multipliers $\bm{\alpha}\in \mathbb{R}^N,\bm{\beta}\in \mathbb{R}^N,\bm{\lambda}\in \mathbb{R}^N$, the Lagrangian dual of the regularized MR-OPT model \cref{MRM} is
\begin{multline}\label{Langrange}
    L(\bm{\gamma},\bm{z},\bm{w},\bm{\alpha},\bm{\beta},\bm{\lambda}) = \sum_{i=1}^{N}\sum_{j=1}^{N} (C_{ij}\gamma_{ij}+\varepsilon \gamma_{ij} \ln(\gamma_{ij})+\varepsilon (\eta_{ij}-\gamma_{ij}) \ln(\eta_{ij}-\gamma_{ij})) \\
	 +\sum_{i=1}^{N} \varepsilon z_i \left(\ln z_i-1 \right)+
		\sum_{i=1}^{N} \varepsilon w_i (\ln w_i-1)
        +\sum_{j=1}^{N}\lambda_j\left(\sum_{i=1}^{N}\gamma_{ij}-v_j\right) \\
         +\sum_{i=1}^{N}\beta_i\left(\sum_{j=1}^{N}\gamma_{ij}-w_i-\hat{u}_i\right) + \sum_{i=1}^{N}\alpha_i\left(\sum_{j=1}^{N}\gamma_{ij}+z_i-\overline{u}_i\right). 
\end{multline}
According to the first-order KKT condition, 
we obtain the following functions with auxiliary variables $\bm{\rho}\in \mathbb{R}^{N\times N},\bm{\phi}\in \mathbb{R}^N,\bm{\xi}\in \mathbb{R}^N$, and $\bm{\psi}\in \mathbb{R}^N$
\begin{equation}\label{Langrange-1}
	\begin{aligned}
		\rho_{ij}=\phi_i \xi_i K_{ij} \psi_j, \
		\gamma_{ij}=\frac{\rho_{ij}\eta_{ij}}{1+\rho_{ij}},\
		z_i=\phi_i,\
		w_i=\xi_i^{-1}, 
	\end{aligned}
\end{equation}	
where 
\begin{equation*}
    \phi_i=e^{-\alpha_i / \varepsilon},\ \xi_i=e^{-\beta_i / \varepsilon}, \ K_{ij}=e^{- C_{ij} / \varepsilon}, \ \psi_j=e^{- \lambda_j / \varepsilon}. 
\end{equation*}

Substituting \cref{Langrange-1} into the Lagrangian in
 \cref{Langrange} and setting $\dfrac{\partial L}{\partial \alpha_i}=0$, $\dfrac{\partial L}{\partial \beta_i}=0$, and $\dfrac{\partial L}{\partial \lambda_j}=0$, we have
$3N$ single-variable equations
\begin{align}
    f_i(\phi_i) &\triangleq \sum_{j=1}^{N}\frac{\eta_{ij}}{1+\phi_i \xi_i K_{ij} \psi_j}-\sum_{j=1}^{N}\eta_{ij}-\phi_i +\overline{u}_i =0, \ i=1,2,\cdots,N, \label{eq-phi} \\
    g_i(\xi_i) &\triangleq \sum_{j=1}^{N}\frac{\eta_{ij}}{1+\phi_i \xi_i K_{ij} \psi_j}-\sum_{j=1}^{N}\eta_{ij}+\xi_i^{-1} +\hat{u}_i =0, \ i=1,2,\cdots,N, \label{eq-xi} \\
    h_j(\psi_j) &\triangleq \sum_{i=1}^{N}\frac{\eta_{ij}}{1+\phi_i \xi_i K_{ij} \psi_j}-\sum_{i=1}^{N}\eta_{ij}+v_j =0, \  j=1,2,\cdots,N, \label{eq-psi}
\end{align}
whose solutions 
are the optimal variables $\bm{\phi}$, $\bm{\xi}$ and $\bm{\psi}$ in the Lagrangian dual. 
Thus, the optimal solution $\bm{\gamma}$, $\bm{z}$, $\bm{w}$ of the regularized MR-OPT model \cref{MRM} can be obtained by finding the
solutions of equations \cref{eq-phi,eq-xi,eq-psi}.

Consider auxiliary variables
\begin{equation*}
    \varsigma_i = (\overline{u}_i - \hat{u}_i ) \xi_i - 1,  
\end{equation*}
and we find that
\begin{equation*}
    \phi_i=\frac{\varsigma_i(\overline{u}_i-\hat{u_i})}{\varsigma_i+1}, \  \xi_i=\frac{\varsigma_i+1}{\overline{u}_i-\hat{u_i}}, \  \rho_{ij}=\varsigma_i K_{ij} \psi_j, 
\end{equation*}
Therefore, solving 
\cref{eq-phi,eq-xi,eq-psi} is equivalent to 
solving 
the following equations
\begin{align}
    \hat{f}_i(\varsigma_i) &\triangleq  \sum_{j=1}^{N}\frac{\eta_{ij}}{1+\varsigma_i K_{ij} \psi_j}-\sum_{j=1}^{N} \eta_{ij}+\frac{\overline{u_i}-\hat{u_i}}{\varsigma_i+1} +\hat{u}_i =0, \ i=1,2,\cdots,N, \label{eq-x} \\
    \hat{h}_j(\psi_j) &\triangleq  \sum_{i=1}^{N}\frac{\eta_{ij}}{1+\varsigma_i K_{ij} \psi_j}- \sum_{i=1}^{N} \eta_{ij}+v_j =0, \ j=1,2,\cdots,N. \label{eq-psi2}
\end{align}
With the elimination of redundant variables $\bm{\varsigma}$, the regularized MR-OPT model \cref{MRM} can be calculated by only solving $2N$ single-variable equations.

One can find that
\begin{equation*}
\begin{aligned}
	&\hat{f}_i'(\varsigma_i)= -\sum_{j=1}^{N}\frac{\eta_{ij}K_{ij} \psi_j}{(1+\varsigma_i K_{ij} \psi_j)^2}-\frac{\overline{u_i}-\hat{u_i}}{(\varsigma_i+1)^2}<0, \\
    &\hat{h}_j'(\psi_j)=  -\sum_{j=1}^{N}\frac{\eta_{ij}\varsigma_i K_{ij}}{(1+\varsigma_i K_{ij} \psi_j)^2}<0, 
\end{aligned}
\end{equation*}
which means that these functions are monotonically decreasing.
Moreover,
\begin{equation*}
\begin{aligned}
	&\lim_{\varsigma_i\rightarrow 0^+}=\overline{u}_i>0, \ \lim_{\varsigma_i\rightarrow \infty}=-\sum_{j=1}^{N} \eta_{ij} +\hat{u}_i<0, \\
 &\lim_{\psi_j\rightarrow 0^+}=\overline{v}_i>0,\ \lim_{\psi_j\rightarrow \infty}=-\sum_{i=1}^{N} \eta_{ij} +v_j<0,  
\end{aligned}
\end{equation*}
which implies that each $\hat{f}_i$ (or $\hat{h}_j$) has a unique positive zero point. 
Given that
$\hat{f}_i''(\varsigma_i)> 0$ (or $\hat{h}_j''(\psi_j)>0$), each zero point can be obtained by Newton's method with global convergence \cite{kincaid2009numerical}.

In our proposed algorithm, the Lagrangian dual variables are updated alternately, which is similar to the Sinkhorn algorithm for the classical OT model. However, due to the extra constraints on the marginals and transport plan, unlike the matrix-vector multiplication in the Sinkhorn algorithm, each alternate iteration step is to solve several single-variable nonlinear equations. Thus it is named the Generalized Alternating Sinkhorn (GAS) algorithm. 
The pseudo-code of the Generalized Alternating Sinkhorn-I (GAS-I) algorithm is given in \cref{alg:MRM}. 

\begin{sloppypar}
\begin{remark}
    It is worth noting that in \cite{wu2023communication} and \cite{chen2023information}, 
    there exist Alternating Sinkhorn~(AS) and Generalized Alternating Sinkhorn~(GAS) algorithms for information theory problems interpreted into optimal transport. Both their and our algorithms leverage an alternating manner similar to the Sinkhorn algorithm, while much more extra constraints in our OPT model lead to nonlinear equations solving in each iteration step. 
\end{remark}
\end{sloppypar}
\renewcommand{\thealgorithm}{I}
\begin{algorithm}[t]
	\renewcommand{\algorithmicrequire}{\textbf{Input:}}
	\renewcommand{\algorithmicensure}{\textbf{Output:}}
		\caption{The GAS-I Algorithm for the OPT Model}
		\label{alg:MRM}
            \begin{spacing}{1.2}
		\begin{algorithmic}[1]
			\REQUIRE $\bm{\hat{u}},\bm{\overline{u}},\bm{v},\bm{C},N,\varepsilon,\bm{\eta},\bm{K}, \delta,L$
			\ENSURE $\bm{\gamma}, \bm{z}, \bm{w}, W$
            \IF{$\sum_{j=1}^{N}\eta_{ij}<\hat{u}_i$ or $\overline{u}_i<\hat{u}_i$ or $\sum_{i=1}^{N}
            \eta_{ij}<v_j$}
            \STATE {\bf return} ``no feasible solution"
            \ENDIF
			\STATE $\bm{\varsigma}\leftarrow \bm{1}_N/N,\quad \bm{\psi}\leftarrow  \bm{1}_N/N,\quad l\leftarrow  1, \quad W\leftarrow 0$
			\WHILE {$l \leq L$}
			\STATE  Solve equations in \cref{eq-x} with Newton's method. ($i=1,2,\cdots ,N$)
			\STATE  Solve equations in \cref{eq-psi2} with Newton's method. ($j=1,2,\cdots ,N$)
			\STATE $l\leftarrow l+1$
			\ENDWHILE
			\FOR {$i$ from $1$ to $N$}
            \STATE $z_i\leftarrow \dfrac{\varsigma_i(\overline{u}_i-\hat{u}_i)}{\varsigma_i+1},\quad w_i \leftarrow  \dfrac{\overline{u}_i-\hat{u}_i}{\varsigma_i+1}$
			\FOR{$j$ from $1$ to $N$}
			\STATE $\gamma_{ij}\leftarrow \dfrac{\varsigma_i K_{ij} \psi_j\eta_{ij}}{1+\varsigma_i K_{ij} \psi_j}, \quad W\leftarrow  W+C_{ij}\gamma_{ij}$
			\ENDFOR
			\ENDFOR
            \RETURN $\bm{\gamma}, \bm{z}, \bm{w}, W$
		\end{algorithmic}
            \end{spacing}
\end{algorithm}

The computation for each $\hat{f}_i$ (or $\hat{h}_j$) requires $O(N)$ time because there are $N$ terms in each equation. Consequently, the time complexity of GAS-I is $O(N^2)$, which is the same as the Sinkhorn algorithm for the classical OT.  
Considering extra constraints on marginals and transport plan in the OPT model, the same complexity as the Sinkhorn algorithm is favorable.

In addition, the log-domain stabilization technique \cite{chizat2018scaling}
is useful for our GAS-I in addressing the severe numerical issues caused by small $\varepsilon$. 
The idea is that when the maximum entry of $\bm{\varsigma}$ or $\bm{\psi}$ exceeds a given threshold $\tau$, these two vectors will be normalized to $\bm{1}_N$ and the excessive part will be absorbed in $\bm{\alpha}$ and $\bm{\beta}$
\begin{align}
    &\bm{\alpha} \leftarrow \bm{\alpha} +\varepsilon \ln(\bm{\varsigma}), \ \bm{\beta} \leftarrow   \bm{\beta} +\varepsilon \ln(\bm{\psi}), \label{log-domain1} \\
    &\bm{K} \leftarrow \operatorname{diag}\left(e^{\bm{\alpha} / \varepsilon}\right) \bm{K} \operatorname{diag}\left(e^{\bm{\beta} / \varepsilon}\right), \
     \bm{\varsigma}\leftarrow \bm{1}_N, \ \bm{\psi}\leftarrow  \bm{1}_N.\label{log-domain2}
\end{align}
After this, we update $\bm{\hat{f}}$ and $\bm{\hat{h}}$ as follows 
\begin{align}
    &\hat{f}_i(\varsigma_i) \triangleq \sum_{j=1}^{N}\frac{\eta_{ij}}{1+\varsigma_i e^{-\alpha_i / \varepsilon} K_{ij} \psi_j}-\sum_{j=1}^{N} \eta_{ij}+\frac{\overline{u}_i-\hat{u}_i}{\varsigma_i+1} +\hat{u}_i=0,\ i=1,2,\cdots ,N,\label{eq-log1} \\
    &\hat{h}_j(\psi_j) \triangleq \sum_{i=1}^{N}\frac{\eta_{ij}}{1+\varsigma_i K_{ij} \psi_j e^{-\beta_j / \varepsilon}} - \sum_{i=1}^{N} \eta_{ij}+v_j=0, \ j=1,2,\cdots ,N.\label{eq-log2} 
\end{align}
We just need to add the steps \cref{log-domain1,log-domain2} after line 8 in \cref{alg:MRM} and use the equations \cref{eq-log1,eq-log2} to replace the equations \cref{eq-x,eq-psi2} in lines 5-6. 
It is worth emphasizing that the utilization of the log-domain stabilization technique does not impact the time complexity of \cref{alg:MRM}.

\subsection{GAS-II Algorithm}\label{subsec:GAS-II}
Note that in the MR-OPT model~\cref{eq-relax}, the total mass $\bm{1}_N^T\bm{\gamma}\bm{1}_N=\bm{1}_N^T\bm{v}$ is fixed, which implies that the following two constraints on the slack variables $\bm{z}$ and $\bm{w}$
\begin{align*}
    \sum_{i=1}^N z_{i} = k_1,\  \sum_{i=1}^N w_{i} = k_2,
\end{align*}
are always hold, where
\begin{align*}
    k_1=\sum_{i=1}^N \overline{u}_i - \sum_{i=1}^{N}\sum_{j=1}^{N} \gamma_{ij},\ k_2= \sum_{i=1}^{N}\sum_{j=1}^{N} \gamma_{ij}-\sum_{i=1}^N \hat{u}_i.
\end{align*}
Then we add these two constraints on the slack variables to the regularized MR-OPT model \cref{MRM}
\begin{equation}\label{con_MRM}
\begin{aligned}
& \mathtoolsset{multlined-width=0.8\displaywidth}
\begin{multlined}
    \min_{\bm{\gamma}, \bm{z}, \bm{w}} \quad \sum_{i=1}^{N}\sum_{j=1}^{N} \left(C_{ij}\gamma_{ij}+\varepsilon \gamma_{ij} \ln(\gamma_{ij})+\varepsilon (\eta_{ij}-\gamma_{ij}) \ln(\eta_{ij}-\gamma_{ij})\right) \\
    +\sum_{i=1}^{N} \varepsilon z_i (\ln z_i-1)+\sum_{i=1}^{N} \varepsilon w_i (\ln w_i-1)
\end{multlined}
\\
& \text{s.t.} \quad \sum_{j=1}^{N} \gamma_{ij} + z_i = \overline{u}_i, \ \sum_{j=1}^{N} \gamma_{ij} - w_i = \hat{u}_i, \ \sum_{i=1}^N \gamma_{ij} = v_j, \ \sum_{i=1}^N z_{i} = k_1,\  \sum_{i=1}^N w_{i} = k_2.
\end{aligned}
\end{equation}

The Lagrangian function $\hat{L}(\bm{\gamma},\bm{z},\bm{w},\bm{\alpha},\bm{\beta},\bm{\lambda},a, b)$ is defined as 
\begin{equation}\label{con_Langrange}
    L(\bm{\gamma},\bm{z},\bm{w},\bm{\alpha},\bm{\beta},\bm{\lambda})
    +a\left(\sum_{i=1}^{N} z_i-k_1\right)+b\left(\sum_{i=1}^{N} w_i-k_2\right), 
\end{equation}
and $L(\bm{\gamma},\bm{z},\bm{w},\bm{\alpha},\bm{\beta},\bm{\lambda})$ is defined in \cref{Langrange}.

Same as the above analysis for GAS-I, finding the solutions of \cref{eq-log1,eq-log2} is equivalent to finding the solutions of equations \cref{eq:s,eq:psiT,eq:TT} in a three-step alternate iterative process
\begin{align}
    \overline{f}_i(s_i) &\triangleq \sum_{j=1}^{N}\frac{\eta_{ij}}{1+s_i K_{ij} \psi_j}-\sum_{j=1}^{N} \eta_{ij}+\frac{\overline{u_i}-\hat{u_i}}{T s_i+1} +\hat{u}_i=0,\ i=1,2,\cdots ,N,\label{eq:s}\\
    \overline{h}_j(\psi_j) &\triangleq \sum_{i=1}^{N}\frac{\eta_{ij}}{1+s_i K_{ij} \psi_j}- \sum_{i=1}^{N} \eta_{ij}+v_j=0,\ j=1,2,\cdots ,N,\label{eq:psiT}\\
    \overline{h}_{N+1}(T) &\triangleq \sum_{i=1}^{N} \frac{\overline{u_i}-\hat{u_i}}{T s_i+1} -k_2=0.\label{eq:TT}
\end{align}

We denote it as the Generalized Alternating Sinkhorn-II~(GAS-II) algorithm. and the pseudo-code is presented in \cref{alg:MRM II}. 

\renewcommand{\thealgorithm}{II}
\begin{algorithm}[t]
	\renewcommand{\algorithmicrequire}{\textbf{Input:}}
	\renewcommand{\algorithmicensure}{\textbf{Output:}}
            \caption{The GAS-II Algorithm for the OPT Model}
		\label{alg:MRM II}
            \begin{spacing}{1.2}
		\begin{algorithmic}[1]
			\REQUIRE $\bm{\hat{u}},\bm{\overline{u}},\bm{v},\bm{C},N,\varepsilon,\bm{\eta},\bm{K}, \delta,L$
			\ENSURE $\bm{\gamma}, \bm{z}, \bm{w}, W$
            \IF{$\sum_{j=1}^{N}\eta_{ij}<\hat{u}_i$ or $\overline{u}_i<\hat{u}_i$ or $\sum_{i=1}^{N}
            \eta_{ij}<v_j$}
            \STATE {\bf return} ``no feasible solution"
            \ENDIF
			\STATE $\bm{s}\leftarrow \bm{1}_N/N,\quad \bm{\psi}\leftarrow  \bm{1}_N/N,\quad T\leftarrow 1, \quad l\leftarrow  1, \quad W\leftarrow 0$
			\WHILE {$l \leq L$}
			\STATE  Solve equations \cref{eq:s} with Newton's method. ($i=1,2,\cdots ,N$)
			\STATE  Solve equations \cref{eq:psiT} with Newton's method. ($j=1,2,\cdots ,N$)
            \STATE  Solve equation \cref{eq:TT} with Newton's method. 
			\STATE $l\leftarrow l+1$
			\ENDWHILE
			\FOR {$i$ from $1$ to $N$}
            \STATE $z_i\leftarrow \dfrac{Ts_i(\overline{u}_i-\hat{u}_i)}{Ts_i+1},\quad w_i \leftarrow  \dfrac{\overline{u}_i-\hat{u}_i}{Ts_i+1}$
			\FOR{$j$ from $1$ to $N$}
			\STATE $\gamma_{ij}\leftarrow \dfrac{s_i K_{ij} \psi_j\eta_{ij}}{1+s_i K_{ij} \psi_j}, \quad W\leftarrow  W+C_{ij}\gamma_{ij}$
			\ENDFOR
			\ENDFOR
            \RETURN $\bm{\gamma}, \bm{z}, \bm{w}, W$
		\end{algorithmic}
            \end{spacing}
\end{algorithm}

Obviously the GAS-I~(Algorithm \ref{alg:MRM}) can be regarded as the GAS-II~(Algorithm \ref{alg:MRM II}) with $T$ set to $1$ in each iteration.
Comparing Algorithm I and Algorithm II, the main difference is the calculation of an additional variable $T$, as indicated in line 7 of Algorithm II.
With merely an additional variable $T$,  the total mass $\bm{1}_N^T\bm{\gamma}\bm{1}_N=\bm{1}_N^T\bm{v}$ is fixed in each iteration, thereby enabling us to complete the convergence proof in next section.
Finally, the time complexity of GAS-II is $O(N^2)$, which is the same as GAS-I.

\begin{remark}
The log-domain stabilization technique \cite{chizat2018scaling} is also useful for our GAS-II
in addressing the severe numerical issues caused by small $\varepsilon$. 
The utilization of the log-domain stabilization technique does not impact the time complexity of GAS-II algorithm. 
\end{remark}

\subsection{Convergence Analysis of GAS-II Algorithm}\label{subsec:convergence}
In this subsection, we prove the convergence and analyze the number of iterations of the proposed GAS-II algorithm inspired by previous works~\cite{wu2022double,villani2009optimal,peyre2019computational}. 
The main results are summarized as follows.

\begin{theorem}\label{thm:con}
Any limit point of the variables $\left(\bm{\alpha}^{(l)},\bm{\beta}^{(l)},\bm{\lambda}^{(l)},a^{(l)},b^{(l)}\right)$ in the iteration of the GAS-II is a coordinatewise maximum point of the dual problem
 \begin{multline}
\max_{\bm{\alpha},\bm{\beta},\bm{\lambda},a,b}  -\varepsilon \sum_{i=1}^{N}\sum_{j=1}^{N} \eta_{ij}\ln(1+\rho_{ij})-\varepsilon \sum_{i=1}^{N} t_1 \phi_{i}\\
  -\varepsilon \sum_{i=1}^{N} t_2 \frac{1}{\xi_{i}} - \sum_{i=1}^{N} \alpha_{i} \overline{u}_i- \sum_{i=1}^{N} \beta_{i} \hat{u}_i-\sum_{j=1}^{N}\lambda_jv_j -ak_1 -b k_2,
\end{multline}
where 
$t_1=e^{-\frac{a}{\varepsilon}}, \  t_2=e^{-\frac{b}{\varepsilon}}.$
\end{theorem}

\begin{theorem}\label{thm:num}
Given any threshold $\delta$ and stopping criterion
    \begin{multline*}
      \left\Vert \bm{\gamma}^{(l)}\bm{1}_N-\bm{\overline{u}}+t_1^{(l)} e^{\bm{x^{(l)}}} \right\Vert_2+ \left\Vert \bm{\gamma}^{(l)}\bm{1}_N-\bm{\hat{u}}-t_2^{(l)} e^{\bm{-y^{(l)}}}\right\Vert_2 + \left\Vert \bm{\gamma}^{(l)^T}\bm{1}_N-\bm{v} \right\Vert_2 \\
      + \left\Vert t_1^{(l)} \bm{1^T}e^{\bm{x^{(l)}}}-k_1 \right\Vert_2 + \left\Vert t_2^{(l)} \bm{1^T}e^{-\bm{y^{(l)}}}-k_2 \right\Vert_2 \leq \delta,
  \end{multline*}
the GAS-II outputs transport plan $\bm{\gamma}$
  within $l$ steps and  
  $$l\leq \left\lceil 4+\frac{8\sqrt{N} LR}{\delta} \right\rceil,$$
  where $L=\max \left\Vert \nabla^2 G(\bm{x},\bm{y},\bm{z}) \right\Vert_2 $.
\end{theorem}

We first show the convergence of GAS-II by rewriting the algorithm into a block coordinate ascent scheme and we present the proof of \cref{thm:con}.
\begin{proof}
The dual of the model \cref{con_MRM} is
\begin{multline}\label{con_dual2}
    \max_{\bm{\alpha},\bm{\beta},\bm{\lambda},a,b} g(\bm{\alpha},\bm{\beta},\bm{\lambda},a,b)
  =\max_{\bm{\alpha},\bm{\beta},\bm{\lambda},a,b}  -\varepsilon \sum_{i=1}^{N}\sum_{j=1}^{N} \eta_{ij}\ln(1+\rho_{ij})-\varepsilon \sum_{i=1}^{N} t_1 \phi_{i}\\
  -\varepsilon \sum_{i=1}^{N} t_2 \frac{1}{\xi_{i}} - \sum_{i=1}^{N} \alpha_{i} \overline{u}_i- \sum_{i=1}^{N} \beta_{i} \hat{u}_i-\sum_{j=1}^{N}\lambda_jv_j -ak_1 -b k_2. 
\end{multline}
Then we have 
\begin{equation}\label{con_partial}
	\begin{aligned}
	\frac{\partial g}{\partial\alpha_{i}}&=\sum_{j=1}^{N}\frac{ \eta_{ij}\rho_{ij}}{1+\rho_{ij}}+t_1 \phi_i-\overline{u}_i,
    \frac{\partial g}{\partial a}=\sum_{i=1}^{N} t_1 \phi_i-k_1, \\
	\frac{\partial g}{\partial \beta_i}&=\sum_{j=1}^{N}\frac{ \eta_{ij}\rho_{ij}}{1+\rho_{ij}}-t_2 \frac{1}{\xi_i}-\hat{u}_i,
    \frac{\partial g}{\partial b}=\sum_{i=1}^{N} t_2 \frac{1}{\xi_i}-k_2,\\
	\frac{\partial g}{\partial \lambda_j}&=\sum_{i=1}^{N}\frac{ \eta_{ij}\rho_{ij}}{1+\rho_{ij}}-v_j.\\
	\end{aligned}
\end{equation}

Solving \cref{con_MRM} with Algorithm \ref{alg:MRM II} is equivalent to
\begin{equation}\label{con_iter}
	\begin{aligned}
	(\bm{\alpha}^{(l+1)})&=\mathop{\mathrm{argmax}}\limits_{\bm{\alpha}}  g(\bm{\alpha},\bm{\beta}^{(l)},\bm{\lambda}^{(l)},a^{(l)},b^{(l)}),\\
	(\bm{\beta}^{(l+1)})&=\mathop{\mathrm{argmax}}\limits_{\bm{\beta}}  g(\bm{\alpha}^{(l+1)},\bm{\beta},\bm{\lambda}^{(l)},a^{(l)},b^{(l)}),\\
	(\bm{\lambda}^{(l+1)},a^{(l+1)}, b^{(l+1)})&=\mathop{\mathrm{argmax}}\limits_{\bm{\lambda},a ,b}  g(\bm{\alpha}^{(l+1)},\bm{\beta}^{(l+1)},\bm{\lambda},a,b),
	\end{aligned}
\end{equation}
which is the iteratiron of the block coordinate ascent. 
Based on the fact that $g$ is continuous,
quasiconcave, and hemivariate, according to the Proposition 5.1 of \cite{tseng2001convergence}, we know that any limit point of
$\left(\bm{\alpha}^{(l)},\bm{\beta}^{(l)},\bm{\lambda}^{(l)},a^{(l)},b^{(l)}\right)$ is a coordinatewise maximizer of $g(\bm{\alpha},\bm{\beta},\bm{\lambda},a,b)$.
\end{proof}

\cref{thm:con} shows that our proposed algorithm converges to a coordinatewise maximum point. Furthermore, we analyze the number of iterations of GAS-II in \cref{thm:num}.
Before the proof of \cref{thm:num}, some lemmas are required. First, we show that there is a threshold of the difference between the maximum and minimum for the variables defined in the following lemma. 
\begin{lemma}\label{lemma:R}
    Let $(\bm{\alpha}^{*},\bm{\beta}^{*},\bm{\lambda}^{*},a^{*},b^{*})$ be the optimal solution of the dual regularized MR-OPT model \cref{con_dual2}. 
    For the sake of simplicity, we denote $x_i =-\alpha_i / \varepsilon, y_i =-\beta_i / \varepsilon$.
    Based on the fact that $\gamma_{ij}>0$,  given a further assumption $\gamma_{ij}\geq \eta_{ij} / c$ as lower bound capacity constraints, where $c$ is a sufficiently large constant so that this assumption is true throughout the algorithm. We have
 $$\max_{i} x^{(l)}_i-\min_{i} x^{(l)}_i\leq R, \quad \max_{i} x^{*}_i-\min_{i} x^{*}_i\leq R,$$
 $$\max_{i} y^{(l)}_i-\min_{i} y^{(l)}_i\leq R, \quad \max_{i} y^{*}_i-\min_{i} y^{*}_i\leq R,$$
 where
  \begin{multline*}
     R= \ln b -\ln a-\ln v +\ln c + \ln\left(\max\left\{\max_{i} p_i+\max_i \Delta u_i, \max_i q_i\right\}\right) \\
     -\ln\left(\min\left\{\min_{i} p_i,\min_i(q_i-\Delta u_i)\right\}\right), 
 \end{multline*}
 \begin{gather*}
     a=\min_{ij} \eta_{ij}, \  b=\max_{ij} \eta_{ij}, \ v=\min_{ij} K_{ij},\\
     \ \Delta u_i=\overline{u}_i-\hat{u}_i, \
     p_i=\sum_{j=1}^{N}\eta_{ij}-\overline{u}_i, \ q_i=\sum_{j=1}^{N}\eta_{ij}-\hat{u}_i. 
 \end{gather*}
 Here, we assume that $p_i > q_i > 0, r_j>0$. If $p_i \leq 0$, note that $\sum_{j=1}^{N}\gamma_{ij}\leq \sum_{j=1}^{N}\eta_{ij}$, so $\overline{u}_i$ can be replaced by $\sum_{j=1}^{N}\eta_{ij}$. If $r_j = 0$, it means that $\gamma_{ij} = \eta_{ij}$ for $i= 1,\cdots,N$. Thus, these variables $\gamma_{ij} (i=1,\cdots,N)$ are fixed and can be removed. 
 
\end{lemma}
\begin{proof}
Firstly, we prove that $\max_{i} x_i-\min_{i} x_i\leq R$.
Note that
$$ \overline{u}_i=\sum_{j=1}^{N}\frac{\eta_{ij}\rho_{ij}}{1+\rho_{ij}}+t_1 e^{x_i}=\sum_{j=1}^{N}\left(\eta_{ij}-\frac{\eta_{ij}}{1+e^{x_i}e^{y_i} K_{ij} e^{z_j}}\right)+t_1 e^{x_i}. $$
With the assumption $\gamma_{ij} \geq \eta_{ij} / c$, it follows
$$ \frac{1}{c\rho_{ij}}\leq \frac{1}{1+\rho_{ij}}, \quad \sum_{j=1}^{N} \frac{\eta_{ij}}{c\rho_{ij}}\leq  \sum_{j=1}^{N} \frac{\eta_{ij}}{1+\rho_{ij}}=\sum_{j=1}^{N}\eta_{ij}-\overline{u}_i+t_1 e^{x_i}=p_i+t_1 e^{x_i},$$
and then
\begin{equation*}
    e^{x_i} \geq \frac{1}{c(p_i+t_1 e^{x_i})}\sum_{j=1}^{N}\frac{\eta_{ij}}{e^{y_i} K_{ij} e^{z_j}} \geq \frac{a}{c(\max_{i} p_i+\max_i \Delta u_i)}\sum_{j=1}^{N}\frac{1}{e^{y_i} e^{z_j}}. 
\end{equation*}
That is
$$\min_{i} x_i \geq \ln a -\ln c -\ln\left(\max_{i} p_i+\max_i \Delta u_i\right)+\ln\left(\sum_{j=1}^{N}\frac{1}{e^{y_i} e^{z_j}}\right). $$

Now we calculate the upper of $x_i$. 
$$\sum_{j=1}^{N} \frac{\eta_{ij}}{\rho_{ij}}>  \sum_{j=1}^{N} \frac{\eta_{ij}}{1+\rho_{ij}}=\sum_{j=1}^{N}\eta_{ij}-\overline{u}_i+t_1 e^{x_i}=p_i+t_1 e^{x_i},$$
it signifies that
$$\max_{i} x_i \leq \ln b -\ln v -\ln\min_{i} p_i+\ln\left(\sum_{j=1}^{N}\frac{1}{e^{y_i} e^{z_j}}\right). $$
Therefore
$$\max_{i} x_i-\min_{i} x_i\leq \ln b -\ln a-\ln v +\ln c + \ln\left(\max_{i} p_i+\max_i \Delta u_i\right)-\ln\min_{i} p_i\leq R. $$
Similarly, we have
\begin{equation*}
\begin{aligned}
    \max_{i} y_i-\min_{i} y_i \leq & \ln b -\ln a-\ln v +\ln c + \ln\max_{i} q_i-\ln\left(\min_i(q_i-\Delta u_i)\right)\leq R.
\end{aligned}
\end{equation*}

The proof of this result for $y_i$ is quite similar to that given for $x_i$ and so is omitted.
\end{proof}

Next, we use this threshold to control the difference between the objective value in the calculation process and the optimal value.
\begin{lemma}
The dual problem can be rewritten as
 $$\min_{\bm{x},\bm{y},\bm{z},t_1,t_2}\quad G(\bm{x},\bm{y},\bm{z},t_1,t_2),$$
 where
 \begin{multline*}
          G(\bm{x},\bm{y},\bm{z},t_1,t_2)= \sum_{i=1}^{N}\sum_{j=1}^{N} \eta_{ij}\ln\left(1+\rho_{ij}\right)+\sum_{i=1}^{N} t_1 e^{x_{i}}+ \sum_{i=1}^{N} t_2 e^{-y_{i}}\\
          -\sum_{i=1}^{N} x_{i} \overline{u}_i- \sum_{i=1}^{N} y_{i} \hat{u}_i-\sum_{j=1}^{N}z_j v_j -k_1 \ln t_1 -k_2 \ln t_2, 
 \end{multline*}
where $z_j$ is defined as $-\lambda_j / \varepsilon$. 

Let $\bm{\Lambda} = \left[ \bm{x},\bm{y},\bm{z},t_1,t_2 \right] \in \mathbb{R}^{3N + 2}$. $\bm{\Lambda}^{*}$ and $\bm{\Lambda}^{(l)}$ are the optimal solution and variables in the $l$-th iteration, respectively. Denote
 $$\widetilde{G}(\bm{\Lambda})=G(\bm{\Lambda})-G(\bm{\Lambda}^{*}),$$
 and then we have
 \begin{equation*}
    \widetilde{G}(\bm{\Lambda}^{(l)})\leq 
    R \sqrt{N}\left( \ \left\Vert \bm{\gamma}^{(l)}\bm{1}_N-\bm{\overline{u}}+t_1 ^{(l)}e^{\bm{x^{(l)}}}\right\Vert_2
    +\left\Vert \bm{\gamma}^{(l)}\bm{1}_N-\bm{\hat{u}}-t_2^{(l)} e^{\bm{-y^{(l)}}}\right\Vert_2\right). 
\end{equation*}
\end{lemma}
\begin{proof}
With the convexity of the function $G$, 
\begin{equation*}
    \widetilde{G}\left(\bm{\Lambda}^{(l)}\right)
    \leq \left\langle \bm{\Lambda}^{(l)} -\bm{\Lambda}^*, \nabla_{\bm{\Lambda}}\widetilde{G}\left(\bm{\Lambda}^{(l)}\right) \right\rangle, 
\end{equation*}
where $\nabla_{\bm{\Lambda}}\widetilde{G}\left(\bm{\Lambda}^{(l)}\right)$ is the summation of
\begin{multline*}
    \nabla_{\bm{x}}\widetilde{G}\left(\bm{\Lambda}^{(l)}\right) = \bm{\gamma}^{(l)}\bm{1}_N-\bm{\overline{u}}+t_1^{(l)} e^{\bm{x^{(l)}}}, \\ 
    \nabla_{\bm{y}}\widetilde{G}\left(\bm{\Lambda}^{(l)}\right) = \bm{\gamma}^{(l)}\bm{1}_N-\bm{\hat{u}}-t_2^{(l)} e^{\bm{-y^{(l)}}}, \
    \nabla_{\bm{z}}\widetilde{G}\left(\bm{\Lambda}^{(l)}\right) = \bm{\gamma}^{(l)^T}\bm{1}_N-\bm{v}, \\
    \nabla_{t_1}\widetilde{G}\left(\bm{\Lambda}^{(l)}\right) = \bm{1}_N^{T}e^{\bm{x}^{(l)}}-\frac{k_1}{t_1^{(l)}}, \
    \nabla_{t_2}\widetilde{G}\left(\bm{\Lambda}^{(l)}\right) = \bm{1}_N^{T}e^{-\bm{y}_i^{(l)}}-\frac{k_2}{t_2^{(l)}}. 
\end{multline*}

Since when the $l$-th iteration is completed, 
$$ \bm{\gamma}^{(l)^{T}} \bm{1}_N=\bm{v}, \ \bm{1}_N^{T}e^{\bm{x}^{(l)}}=\frac{k_1}{t_1^{(l)}}, \ \bm{1}_N^{T}e^{-\bm{y}^{(l)}}=\frac{k_2}{t_2^{(l)}},$$ 
we have that $\left\langle \bm{1}_N, \bm{\gamma}^{(l)}\bm{1}_N-\bm{\overline{u}}+t_1^{(l)} e^{\bm{x^{(l)}}} \right\rangle = \left\langle \bm{1}_N, \bm{\gamma}^{(l)}\bm{1}_N-\bm{\hat{u}}-t_2^{(l)} e^{\bm{-y^{(l)}}} \right\rangle=0$. 
Taking $a=\dfrac{\max_i x_i^{(l)}+\min_i x_i^{(l)}}{2}$, by H{\"o}lder's inequality and \cref{lemma:R}, we obtain
\begin{multline*}
         \left\langle \bm{x}^{(l)}, \bm{\gamma}^{(l)}\bm{1}_N-\bm{\overline{u}}+t_1^{(l)}e^{\bm{x^{(l)}}} \right\rangle= \left\langle \bm{x}^{(l)} -a\bm{1}_N, \bm{\gamma}^{(l)}\bm{1}_N-\bm{\overline{u}}+t_1^{(l)}e^{\bm{x^{(l)}}} \right\rangle \\
    \leq \left\Vert \bm{x}^{(l)} -a\bm{1}_N \right\Vert_{\infty} \left\Vert \bm{\gamma}^{(l)}\bm{1}_N-\bm{\overline{u}}+t_1^{(l)}e^{\bm{x^{(l)}}} \right\Vert_1 \\
    = \frac{\max_i x_i^{(l)}-\min_i x_i^{(l)}}{2} \left\Vert \bm{\gamma}^{(l)}\bm{1}_N-\bm{\overline{u}}+t_1^{(l)}e^{\bm{x^{(l)}}} \right\Vert_1 \\
    \leq \frac{R}{2} \left\Vert \bm{\gamma}^{(l)}\bm{1}_N-\bm{\overline{u}}+t_1^{(l)}e^{\bm{x^{(l)}}} \right\Vert_1
    \leq  \frac{R\sqrt{N}}{2} \left\Vert \bm{\gamma}^{(l)}\bm{1}_N-\bm{\overline{u}}+t_1^{(l)}e^{\bm{x^{(l)}}} \right\Vert_2.  
\end{multline*}
Similarly, 
$$\left\langle \bm{x}^{*}, \bm{\gamma}^{(l)}\bm{1}_N-\bm{\overline{u}}+t_1^{(l)}e^{\bm{x^{(l)}}} \right\rangle\leq \frac{R\sqrt{N}}{2} \left\Vert \bm{\gamma}^{(l)}\bm{1}_N-\bm{\overline{u}}+t_1^{(l)}e^{\bm{x^{(l)}}} \right\Vert_2, $$
and then
$$\left\langle \bm{x}^{(l)}-\bm{x}^{*}, \bm{\gamma}^{(l)}\bm{1}_N-\bm{\overline{u}}+t_1^{(l)}e^{\bm{x^{(l)}}} \right\rangle \leq R\sqrt{N} \left\Vert \bm{\gamma}^{(l)}\bm{1}_N-\bm{\overline{u}}+t_1^{(l)}e^{\bm{x^{(l)}}} \right\Vert_2. $$

As same as the proof for $\bm{x}$, we can prove the same conclusion for $\bm{y}$, and hence this lemma holds. 
\quad
\end{proof}

Based on the above two lemmas, we prove \cref{thm:num} which demonstrates the number of iterations of our GAS-II algorithm. 

\begin{proof}
From the $l$-th iteration to the $(l+1)$-th iteration, $\bm{x}^{(l+1)}$, $\bm{y}^{(l+1)}$, $\bm{z}^{(l+1)}$, $t_1^{(l+1)}$, and $ t_2^{(l+1)}$ are calculated.

We firstly consider $\bm{x}^{(l+1)}$ and $\bm{y}^{(l+1)}$. 
Let
$$\bm{\hat{x}}^{(l+1)}=\bm{x}^{(l)}-\frac{1}{L}\nabla_{\bm{x}}\widetilde{G}\left(\bm{\Lambda}^{(l)}\right), \ \bm{\hat{y}}^{(l+1)}=\bm{y}^{(l)}-\frac{1}{L}\nabla_{\bm{y}}\widetilde{G}\left(\bm{\Lambda}^{(l)}\right).$$ 
Then because $ \widetilde{G}$ is a L-smooth function \cite{beck2017first}, we obtain
\begin{multline*}
    \widetilde{G}\left(\bm{\hat{x}}^{(l+1)},\bm{\hat{y}}^{(l+1)},\bm{z}^{(l)},t_1^{(l)},t_2^{(l)}\right) \\
    \leq  \widetilde{G}\left( \bm{\Lambda}^{(l)} \right) - \frac{1}{2L} \left\Vert \nabla_{\bm{x}}\widetilde{G}\left(\bm{\Lambda}^{(l)}\right) \right\Vert_2^2-\frac{1}{2L}\left\Vert \nabla_{\bm{y}}\widetilde{G}\left(\bm{\Lambda}^{(l)}\right)\right\Vert_2^2.
\end{multline*}
Note that 
$$ \left\Vert \bm{\gamma}^{(l)^T}\bm{1}_N-\bm{v} \right\Vert_2= \left\Vert t_1^{(l)} \bm{1^T}e^{\bm{x^{(l)}}}-k_1 \right\Vert_2 = \left\Vert t_2^{(l)} \bm{1^T}e^{-\bm{y^{(l)}}}-k_2 \right\Vert_2 = 0,$$ 
therefore 
\begin{equation*}
    \left\Vert \bm{\gamma}^{(l)}\bm{1}_N-\bm{\overline{u}}+t_1^{(l)} e^{\bm{x^{(l)}}} \right\Vert_2 + \left\Vert \bm{\gamma}^{(l)}\bm{1}_N-\bm{\hat{u}}-t_2^{(l)} e^{\bm{-y^{(l)}}} \right\Vert_2 \geq \delta 
\end{equation*}
before the iteration terminates. Subsequently 
\begin{equation*}
    \begin{aligned}
        & \widetilde{G}\left(\bm{x}^{(l)},\bm{y}^{(l)}, \bm{z}^{(l)}, t_1^{(l)},t_2^{(l)} \right) -\widetilde{G}\left(\bm{x}^{(l+1)},\bm{y}^{(l+1)},\bm{z}^{(l)},t_1^{(l)},t_2^{(l)}\right) \\
        \geq & \  \widetilde{G}\left(\bm{x}^{(l)},\bm{y}^{(l)},\bm{z}^{(l)},t_1^{(l)},t_2^{(l)}\right)-\widetilde{G}\left(\bm{\hat{x}}^{(l+1)},\bm{\hat{y}}^{(l+1)},\bm{z}^{(l)},t_1^{(l)},t_2^{(l)}\right)\\
        \geq & \ \frac{1}{2L}\left\Vert \nabla_{\bm{x}}\widetilde{G}\left(\bm{x}^{(l)},\bm{y}^{(l)},\bm{z}^{(l)},t_1^{(l)},t_2^{(l)}\right)\right\Vert_2^2+\frac{1}{2L}\left\Vert \nabla_{\bm{y}}\widetilde{G}\left(\bm{x}^{(l)},\bm{y}^{(l)},\bm{z}^{(l)},t_1^{(l)},t_2^{(l)}\right)\right\Vert_2^2\\
        \geq & \ \frac{1}{2 L} \left( \left\Vert \bm{\gamma}^{(l)}\bm{1}_N-\bm{\overline{u}}+t_1^{(l)}  e^{\bm{x^{(l)}}} \right\Vert_2 \right)^2
        +\frac{1}{2 L}\left(\left\Vert \bm{\gamma}^{(l)}\bm{1}_N-\bm{\hat{u}}-t_2 ^{(l)}e^{\bm{-y^{(l)}}}\right\Vert_2\right)^2\\
        \geq & \ \frac{1}{4 L}\left(\left\Vert \bm{\gamma}^{(l)}\bm{1}_N-\bm{\overline{u}}+t_1^{(l)}  e^{\bm{x^{(l)}}}\right\Vert_2
        +\left\Vert \bm{\gamma}^{(l)}\bm{1}_N-\bm{\hat{u}}-t_2^{(l)} e^{\bm{-y^{(l)}}}\right\Vert_2\right)^2.
    \end{aligned}
\end{equation*}
Thus we have
\begin{multline*}
    \widetilde{G}\left(\bm{x}^{(l)},\bm{y}^{(l)},\bm{z}^{(l)},t_1^{(l)},t_2^{(l)}\right)-\widetilde{G}\left(\bm{x}^{(l+1)},\bm{y}^{(l+1)},\bm{z}^{(l)},t_1^{(l)},t_2^{(l)}\right)\\
    \geq \max\left\{\frac{1}{4NLR^2}\widetilde{G}(\bm{x}^{(l)},\bm{y}^{(l)},\bm{z}^{(l)},t_1^{(l)},t_2^{(l)})^2, \frac{\delta^2}{4L}\right\}. 
\end{multline*}

As same as the above proof, we can show that 
$$\widetilde{G}\left(\bm{x}^{(l+1)},\bm{y}^{(l+1)},\bm{z}^{(l)},t_1^{(l)},t_2^{(l)}\right)-\widetilde{G}
\left(\bm{x}^{(l+1)},\bm{y}^{(l+1)},\bm{z}^{(l+1)},t_1^{(l+1)},t_2^{(l+1)}\right)\geq 0.$$
This results in
\begin{multline*}
    \widetilde{G}\left(\bm{x}^{(l)},\bm{y}^{(l)},\bm{z}^{(l)},t_1^{(l)},t_2^{(l)}\right)-\widetilde{G}\left(\bm{x}^{(l+1)},\bm{y}^{(l+1)},\bm{z}^{(l+1)},t_1^{(l+1)},t_2^{(l+1)}\right)\\
    \geq \max\left\{\frac{1}{4NLR^2}\widetilde{G}\left(\bm{x}^{(l)},\bm{y}^{(l)},\bm{z}^{(l)},t_1^{(l)},t_2^{(l)}\right)^2, \frac{\delta^2}{4 L}\right\}. 
\end{multline*}
With the switching strategy in Theorem 1 of \cite{dvurechensky2018computational}, the number of iterations $l$ satisfies
$$l\leq \left\lceil 4+\frac{8 \sqrt{N} LR}{\delta} \right\rceil.$$
\quad
\end{proof}

\section{Numerical Experiments}
\label{sec:numerical}
In this section, we conduct numerical experiments to demonstrate the accuracy and efficiency of our proposed method on both 1D case and 2D case. The solution obtained by Gurobi \cite{optimization2021gurobi} is regarded as the ground truth of the OPT model. 
We extend the Iterative Bregman Projections~(IBP) algorithm \cite{IBP} to the OPT model as a baseline, which is shown in details in Appendix \ref{sec:IBP}.
The accuracy and computational cost of these algorithms are compared in the following experiments.

In the numerical experiments, the regularization parameter $\varepsilon$ is set to $10^{-3}$ for among the IBP, the GAS-I, and the GAS-II algorithms. We use `-' to denote that the computational time exceeds 10800s or the memory required exceeds the RAM size. When Gurobi runs out of memory in large-scale scenarios, we use `N/A' (not available) to represent the relative error due to the lack of ground truth. For each scenario, the computational time and relative error results are averaged over 10 times experiments. All the experiments are conducted in Matlab R2018b on a platform with 448G RAM, Intel(R) Xeon(R) Gold E5-2683v3 CPU @2.00GHz with 14 cores. 

\subsection{1D Case}

In the 1D case, we consider the uniform grid points and generate two discrete distributions on the grid points from the standard uniform distribution on $[0,1]$
$$\bm{u}=(u_1, u_2, \cdots,u_{N}), \quad \bm{v}=(v_1,v_2,\cdots,v_{N}).$$ Then $\bm{u}$ and $\bm{v}$ are normalized so that 
$$\sum_{i=1}^{N}u_i=1,\quad \sum_{j=1}^{N}v_j=1.$$
We examine two types of bounds setting for the marginal distribution $\bm{u}$. One is multiplying a random number to each element of $\bm{u}$, the other is multiplying a constant to all elements of $\bm{u}$.
Without loss of generality, the upper and lower bounds are set in each of the two ways described above as follows
$$\bm{\hat{u}}=\bm{k}\odot \bm{u} = \left(k_1 u_1, \cdots, k_N u_N\right), \quad \bm{\overline{u}}= a\bm{u}=\left( a u_1, \cdots, a u_N\right), $$
where $\bm{k}\in \mathbb{R}^{N}$ is generated from the standard uniform distribution of $[0,1]$, and $a$ is a constant. 
Experimentally, $a$ is set to $1.2$ and $1.5$ in the following part. 
Other upper and lower bound settings are also plausible, such as multiplying $\bm{u}$ by random numbers as the upper bound. In fact, different settings scarcely affect the efficacy of our algorithms compared to the baselines. Due to space limitations, further discussion is omitted.

Referring to \cite{wu2022double}, to guarantee that
the feasible set is non-empty, $2\bm{\overline{u}}\bm{v}^T$ is chosen as the upper bound matrix of $\bm{\gamma}$, since it is easy to find that $\bm{\gamma}^*=\bm{u}\bm{v}^T, \bm{u}^*=\bm{u}$ is one of the feasible solutions. 
Now we employ the Wasserstein-2 metric as the transport cost and set the production cost to be zero, that is 
$$C_{ij}=h^2(i-j)^2,$$
where $h$ is grid spacing.

\begin{table}
\begin{center}
\centering
\subtable[The Average Computational Time]{
\begin{tabular}{c c| c c c c }
	\hline
	\rule[-1ex]{0pt}{2.5ex} \multirow{2}{*}{$a$} &  \multirow{2}{*}{$N$}  &   \multicolumn{4}{c}{times (sec)}  \\ \cline{3-6}  
	   & & GAS-I &GAS-II & IBP & Gurobi \\
    \hline
    \multirow{4}{*}{1.2} & $1000$ &  3.95 & 3.20 & 29.3 & 14.3 \\
    & $2000$  & 15.6& 10.9 &328.0 &68.9 \\
    & $4000$ &76.4 & 52.2&1958.1 & 416.5  \\
    & $8000$  & 311.0& 194.1 & 9734.9 &- \\
    \hline
    \multirow{5}{*}{1.5} & $1000$  & 1.02 & 1.67 & 20.3 & 12.1 \\
    & $2000$&  4.18& 4.64 & 243.4 &  43.5 \\
    & $4000$& 18.6& 19.7 & 1316.8 & 183.2 \\
    & $8000$  & 77.2& 95.5&  6568.1 &800.5 \\
    & $16000$ & 330.5& 522.8&- &   -\\
    \hline       
\end{tabular}
\label{tab1:subtab11}
}
\subtable[The Relative Error]{
\begin{tabular}{c c| c c c}
	\hline
	\rule[-1ex]{0pt}{2.5ex} \multirow{2}{*}{$a$} &  \multirow{2}{*}{$N$}  & \multicolumn{3}{c}{relative error} \\ \cline{3-5}  
	   & &  GAS-I & GAS-II& IBP\\
    \hline
    \multirow{4}{*}{1.2} & $1000$ &  1.40$\times 10^{-3}$ &1.41 $\times 10^{-3}$ &3.44 $\times 10^{-2}$\\
    & $2000$  & 8.51 $\times 10^{-4}$ &8.60 $  \times 10^{-4}$ & 3.00 $  \times 10^{-2}$\\
    & $4000$ &1.42 $\times 10^{-3}$ & 1.44 $\times 10^{-3}$  & 3.38 $\times 10^{-2}$ \\
    & $8000$  &  N/A &N/A  & N/A\\
    \hline
    \multirow{5}{*}{1.5} & $1000$  &   2.69$\times 10^{-3}$&3.26 $\times 10^{-3}$ &6.15 $\times 10^{-3}$\\
    & $2000$& 2.70 $\times 10^{-3}$  &3.26 $  \times 10^{-3}$ &5.70 $  \times 10^{-3}$\\
    & $4000$& 2.72$\times 10^{-3}$ & 3.55$\times 10^{-3}$ &5.59 $\times 10^{-3}$\\
    & $8000$  & 2.75$\times 10^{-3}$  & 3.48$\times 10^{-3}$& 5.81 $\times 10^{-3}$\\
    & $16000$ &N/A & N/A & N/A\\
    \hline       
\end{tabular}
\label{tab1:subtab12}
}
\end{center}
  \caption{The 1D distribution OPT model. The comparison between the IBP, the GAS-I and the GAS-II with different $N$. Table 1(a) is the average computational time and Table 1(b) is the relative error with the ground truth obtained by Gurobi.}  
  \label{tab:1D}    
\end{table}

\begin{figure}[htb] 
    \centering
    \includegraphics[scale=0.4]{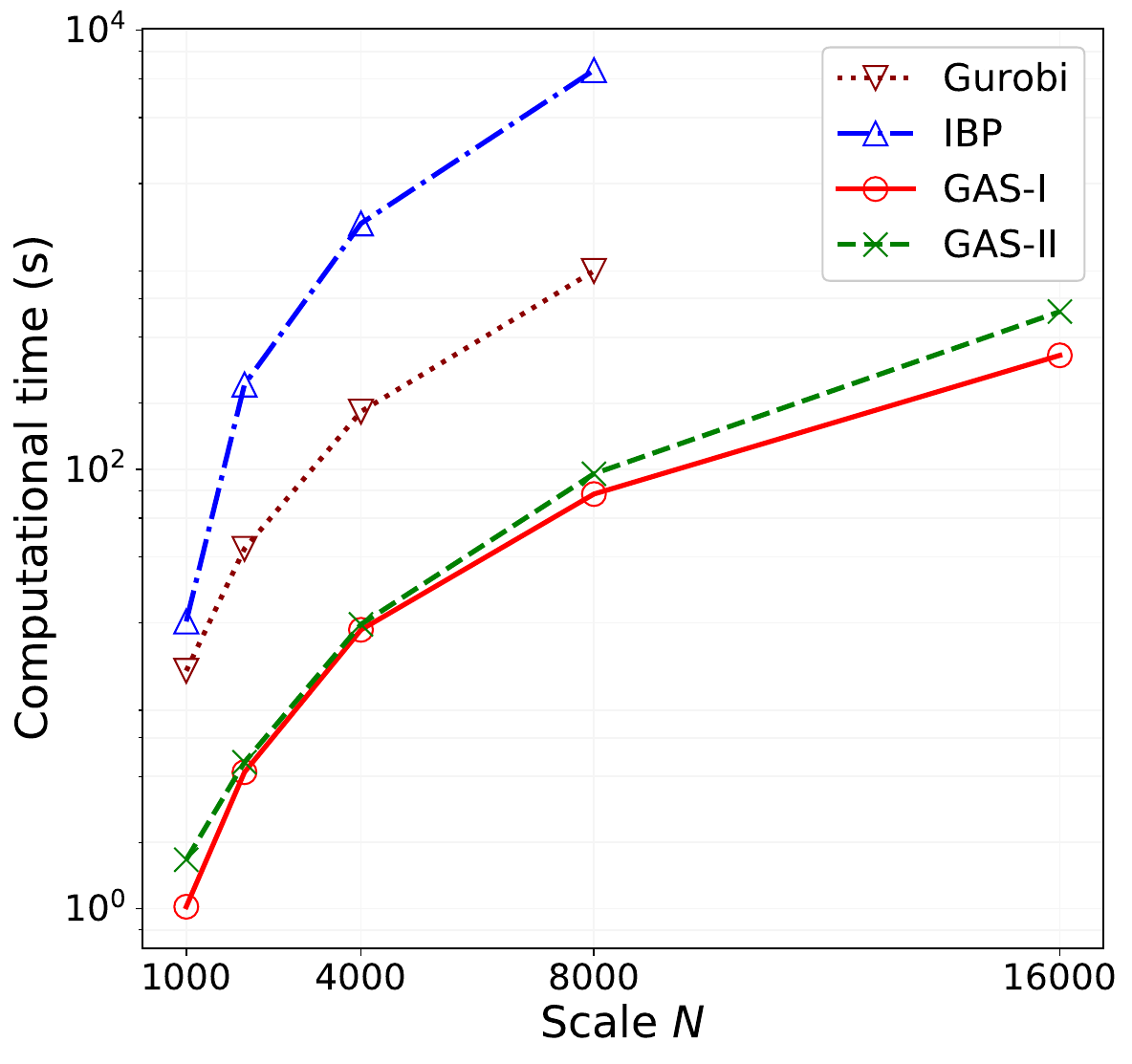}
    \caption{The 1D random distribution optimal production transport problem. The comparison of computational time between Gurobi, IBP, GAS-I, and GAS-II with different numbers of grid points $N$ and the upper bound parameter of $\bm{u}$ is $a=1.5$.}
    \label{fig:time_1D}
\end{figure}

The performance and computational cost of Gurobi, IBP algorithm, GAS-I, and GAS-II with $a=1.2$ and $1.5$ are presented in \cref{tab:1D}.
From \cref{tab:1D}, it is clear that the average computational time of GAS-I and GAS-II are much less than IBP and Gurobi, and we can see that our proposed method achieves up to more than 50 times speedup compared to IBP when $N=2000$ with $a=1.5$. The relative error of GAS-I and GAS-II are significantly smaller than IBP with the ground truth obtained by Gurobi. In \cref{fig:time_1D}, we visually compare the computational time of the three algorithms with the upper bound parameter $a=1.5$. We can see that Gurobi and IBP fail to generate the solutions when $N=16000$ due to time and memory limitations.

\subsection{2D Case}

In the 2D case, we consider the nonuniform grid points and generate two discrete distributions on the grid points from the standard uniform distribution on [0,1] 
$$\bm{u}=(u_{1,1},u_{1,2},\cdots, u_{1,N};\ \cdots;\ u_{N,1},u_{N,2},\cdots,u_{N,N}),$$
$$\bm{v}=(v_{1,1},v_{1,2},\cdots, v_{1,N};\ \cdots;\ v_{N,1},v_{N,2},\cdots,v_{N,N}).$$
Then $\bm{u}$ and $\bm{v}$ are normalized so that 
$$\sum_{i=1}^{N}\sum_{j=1}^{N}u_{ij}=1,\quad \sum_{i=1}^{N}\sum_{j=1}^{N}v_{ij}=1.$$
The coordinates of $\bm{u}$ and $\bm{v}$ on the 2D plane are generated in the following way. Let $(\bm{u}_x,\bm{u}_y)$ (or $(\bm{v}_x,\bm{v}_y)$) be the coordinates of $\bm{u}$ (or $\bm{v}$), and $\bm{u}_x,\bm{u}_y$ (or $\bm{v}_x,\bm{v}_y$) are generated from the standard uniform distribution of $[0,1]^N$. 
The lower bound $\bm{\hat{u}}$ and upper bound $\bm{\overline{u}}$ of $\bm{u}$ are set to
\begin{align*}
    &\bm{\hat{u}}=\bm{k}\odot \bm{u} = \left(k_{1,1}u_{1,1},\cdots, k_{1,N} u_{1,N};\ \cdots;\ k_{N,1} u_{N,1},\cdots,k_{N,N} u_{N,N}\right),\
    \bm{\overline{u}}= 1.5 \bm{u},
\end{align*}
where $\bm{k}\in \mathbb{R}^{N\times N}$ is generated from the standard uniform distribution on $[0,1]$.

Referring to \cite{wu2022double} and trying to guarantee that
the feasible set is non-empty, $\bm{\eta} \in \mathbb{R}^{N^4}, \eta_{i_1,j_1;i_2,j_2} = 2 \overline{u}_{i_1,j_1} \cdot v_{i_2,j_2} $ is chosen as the upper bound matrix of $\bm{\gamma}$.
Now we employ the Wasserstein-2 metric as the transport cost and generate the production cost $\bm{p}\in \mathbb{R}^{N\times N}$ by taking the absolute value of random numbers from the normal distribution $\mathcal{N}(1,0.25)$ with a scale parameter $d\in\{0.1,0.2,0.5\}$, the cost $C_{i_1,j_1;i_2,j_2}$ is
$$C_{i_1,j_1;i_2,j_2}=(u_{x_{i_1}}-v_{x_{i_2}})^2+ (u_{y_{j_1}}-v_{y_{j_2}})^2+d \cdot p_{i_1,j_1},$$
and divided by its maximal entry for normalization.

\begin{table}
\begin{center}
\centering
\subtable[The Average Computational Time]{
\begin{tabular}{c c| c c c c }
	\hline
	\rule[-1ex]{0pt}{2.5ex} \multirow{2}{*}{$d$} &  \multirow{2}{*}{$N$}  &   \multicolumn{4}{c}{times (sec)}  \\ \cline{3-6}  
	   & & GAS-I &GAS-II & IBP & Gurobi \\
    \hline
    \multirow{4}{*}{0.1} & $20\times 20$ &  0.21&0.24  &  2.77 & 3.37 \\
    & $40\times  40$   & 2.81&3.20 & 104.3  &41.3 \\
    & $80 \times  80$ &43.3 &51.5 & 3120.8 &796.3  \\
    & $160 \times  160$ & 827.7& 1008.9  &- & - \\
    \hline
    \multirow{4}{*}{0.2} & $20 \times  20$ &   0.21&0.25  & 3.05 & 2.72  \\
    & $40\times  40$ &  2.87 & 2.63  &104.3  &47.5  \\
    & $80 \times 80$ &  38.3  & 42.7&  3442.9 &493.4   \\
    & $160 \times  160$ &775.9 & 1055.0 & - &- \\
    \hline
    \multirow{4}{*}{0.5} & $20 \times  20$ &  0.19&0.23& 3.13 & 2.95 \\
    & $40 \times  40$ &  2.85&2.86 & 117.1 &44.5 \\
    & $80\times  80$ & 38.4 & 43.0& 3553.7  & 502.4  \\
    & $160\times  160$ & 804.5 & 955.0 &-  & - \\
    \hline       
\end{tabular}
\label{tab2:subtab11}
}
\subtable[The Relative Error]{
\begin{tabular}{c c| c c c}
	\hline
	\rule[-1ex]{0pt}{2.5ex} \multirow{2}{*}{$d$} &  \multirow{2}{*}{$N$}  & \multicolumn{3}{c}{relative error} \\ \cline{3-5}  
	   & &  GAS-I & GAS-II& IBP\\
    \hline
    \multirow{4}{*}{0.1} & $20\times 20$ & 2.75$\times 10^{-4}$ &5.98 $\times 10^{-4}$ &1.76 $\times 10^{-2}$\\
    & $40\times  40$   & 3.27 $\times 10^{-4}$ & 7.52 $  \times 10^{-4}$ &1.64 $  \times 10^{-2}$  \\
    & $80 \times  80$  & 3.40$\times 10^{-4}$    &  6.82 $\times 10^{-4}$  &  1.50 $\times 10^{-2}$ \\
    & $160 \times  160$ & N/A & N/A  & N/A\\
    \hline
    \multirow{4}{*}{0.2} & $20 \times  20$ &   4.56$\times 10^{-4}$  &8.84 $\times 10^{-4}$ \\
    & $40\times  40$ & 5.24 $\times 10^{-4}$  &  1.14 $  \times 10^{-3}$& 1.80 $  \times 10^{-2}$  \\
    & $80 \times 80$  &  5.65$\times 10^{-4}$     &  9.66 $\times 10^{-4}$  &  1.33 $\times 10^{-2}$  \\
    & $160 \times  160$ & N/A  &N/A & N/A\\
    \hline
    \multirow{4}{*}{0.5} & $20 \times  20$  & 5.00$\times 10^{-4}$ & 8.90 $\times 10^{-4}$  &2.22 $\times 10^{-2}$ \\
    & $40 \times  40$ & 5.76 $\times 10^{-4}$ & 1.19 $  \times 10^{-3}$  &1.94 $  \times 10^{-2}$  \\
    & $80\times  80$  &  6.23$\times 10^{-4}$     &   1.11 $\times 10^{-3}$  &   1.44 $\times 10^{-2}$ \\
    & $160\times  160$  & N/A  & N/A  & N/A\\
    \hline    
\end{tabular}
\label{tab2:subtab12}
}
\end{center}
  \caption{The 2D distribution OPT model. The comparison between the IBP, the GAS-I and the GAS-II algorithms with different $N$. Table 2(a) is the average computational time and Table 2(b) is the relative error with the ground truth obtained by Gurobi.}  
  \label{tab:2D}    
\end{table}

\begin{figure}[htb] 
    \centering
    \includegraphics[scale=0.4]{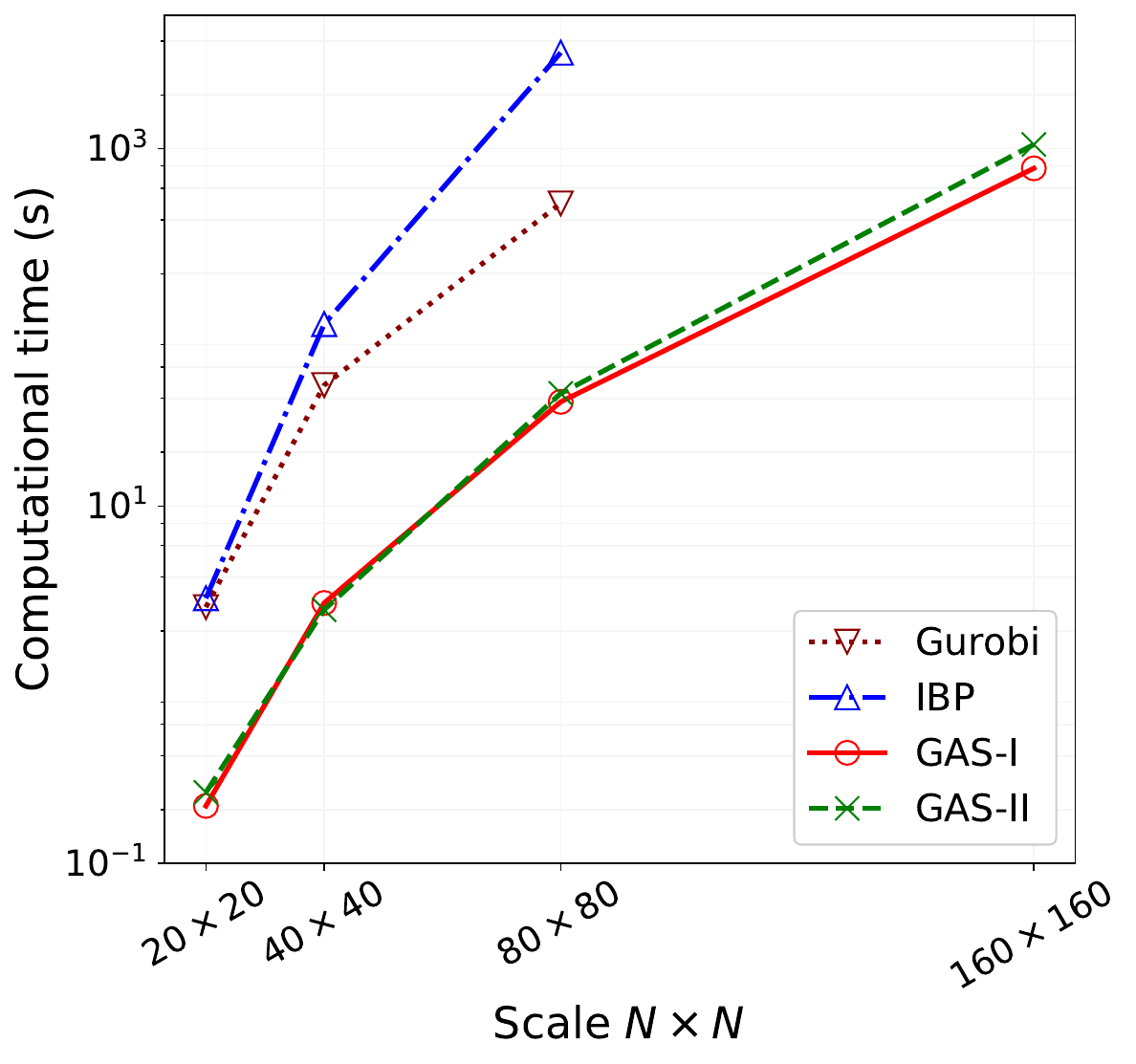}
    \caption{The 2D random distribution OPT model. The comparison of computational time between Gurobi, IBP, GAS-I, and GAS-II with different numbers of grid points $N\times N$ and the parameter of production cost is $d=0.2$.}
    \label{fig:time_2D}
\end{figure}

The performance and computational costs of Gurobi, IBP, GAS-I, and GAS-II with $d=0.1$, $0.2$, and $0.5$ are presented in \cref{tab:2D}.
From \cref{tab:2D}, it is clear that the average computational time of GAS-I and GAS-II are much less than IBP and Gurobi. We can find that the proposed method achieves up to 40 times speedup compared with IBP when $N=40$ with $d=0.2$ and $0.5$. The relative error of GAS-I and GAS-II are overwhelmingly smaller than IBP with the ground truth obtained by Gurobi. 
In \cref{fig:time_2D}, we visually compare the computational time of the three algorithms with the parameter of production cost is $d=0.2$. We can see that Gurobi and IBP fail to generate the solutions when $N=160$ due to time and memory limitations.

\subsection{Comparison of Memory Usage}

In \cref{fig:memory} (upper), we  present the memory usage of IBP, Gurobi, GAS-I and GAS-II with different numbers of grid points $N$ and the upper bound parameter of $\bm{u}$ is $a=1.5$ in the 1D case. We can find that the space complexity of IBP is $O(N^2)$ and it is $O(N)$ of GAS-I and GAS-II, which are both much less than Gurobi.

In \cref{fig:memory} (lower), we  present the memory usage of Gurobi, IBP, GAS-I, and GAS-II with different numbers of grid points $N\times N$ and  the parameter of production cost is $d=0.2$ in the 2D case. Similar to the 1D case, we can find that the space complexity of IBP is $O(N^4)$ and it is $O(N^2)$ of GAS-I and GAS-II, which are both much less than Gurobi.

\begin{figure}[htb] 
	\centering
	\subfigure{
		\includegraphics[scale=0.35]{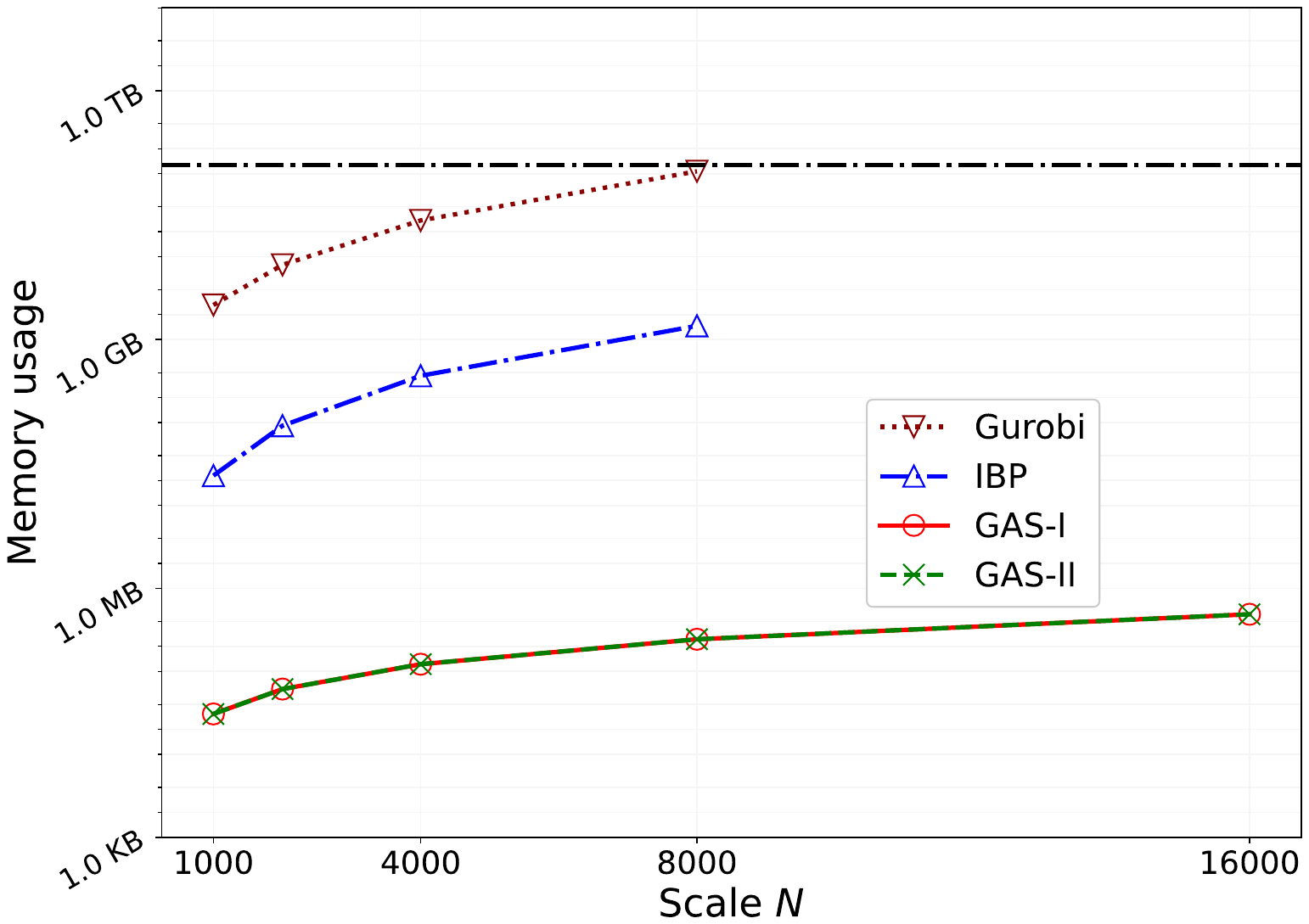}}
	\subfigure{
		\includegraphics[scale=0.35]{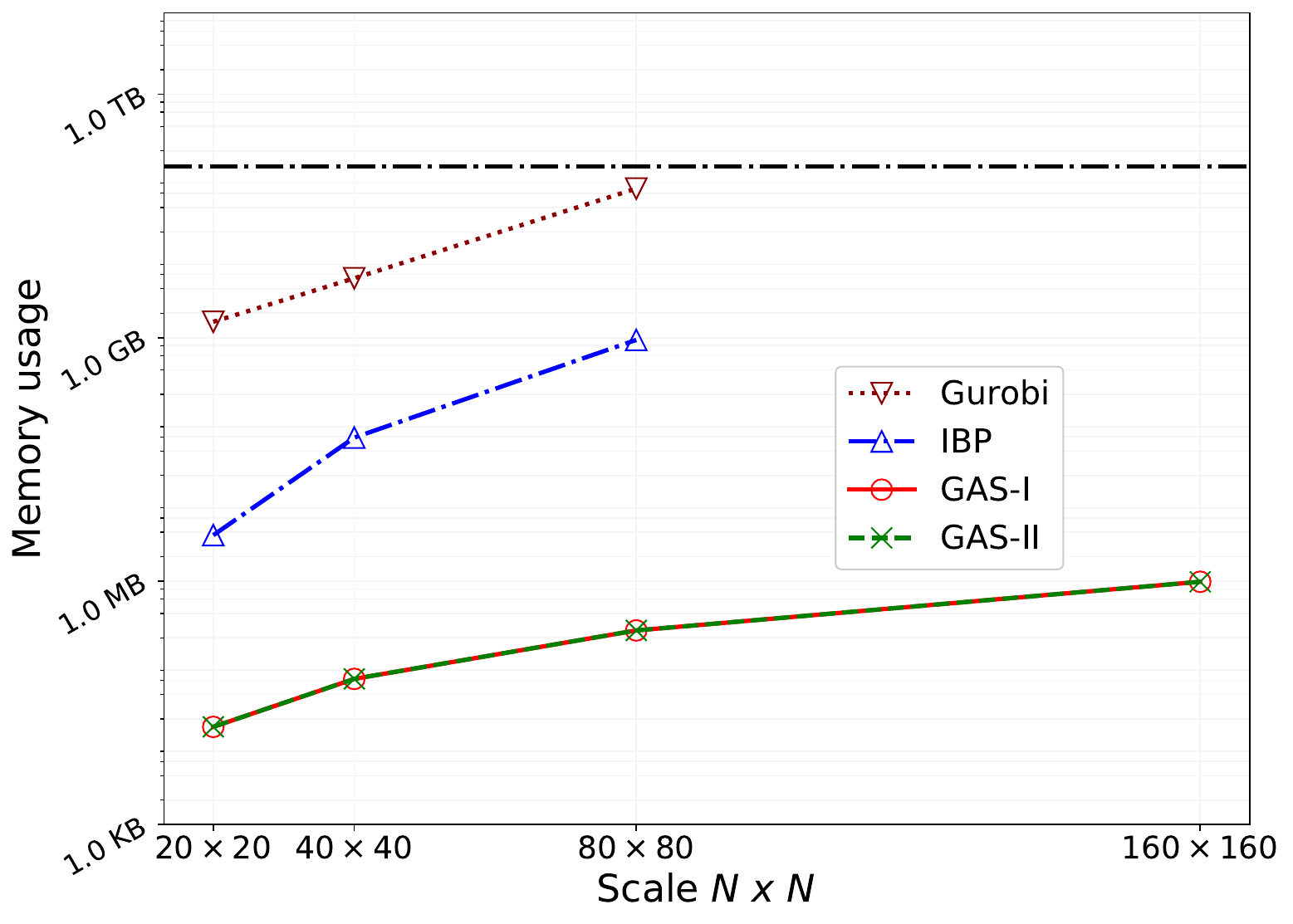}}
	\caption{The memory usage of Gurobi, IBP, GAS-I, and GAS-II in 1D case and 2D case. Upper: The comparison of memory usage among these four algorithms with different numbers of grid points $N\times N$ and the upper bound parameter of $\bm{u}$ is $a=1.5$. Lower: The comparison of memory usage among these four algorithms with different numbers of grid points $N$ and the parameter of production cost is $d=0.2$. }
	\label{fig:memory} 
\end{figure}

\section{Application: Coal Production and Transport}
\label{sec:applicationA}
In this section, we demo-nstrate the potential application of the OPT model in national coal production and transport. Considering the importance of coal in the national economy, several works have applied optimal transport theory to model and solve coal resource allocation problems \cite{kantorovich1960mathematical,cao2006coal,kozan2012demand}. Here, we consider additional production information in the modeling, which is a natural extension.

In the Chinese mainland, the distribution of coal is abundant but severely imbalanced. Specifically, most of the production of coal is concentrated in the western and northern regions, while the consumption of coal is mainly concentrated in the southeastern coastal region \cite{suping2015china,mao1999china}. Therefore, a suitable plan for the massive product and transportation of coal is extremely crucial. 
Based on the distribution of coal and the regulatory capacity of coal production, we model an OPT problem and solve it with our proposed GAS-I algorithm to minimize the cost of production and transport. 
Note that the regional imbalance between production and consumption is also quite common in other regions and resources, so our model and algorithm can be applied to make a plan for production and transport in other scenarios. 

\begin{table}[t]
\renewcommand\arraystretch{1.25}
\begin{scriptsize}
  \centering
    \begin{tabular}{l|rrcrr}
    	\hline
	\rule[-1ex]{0pt}{2.5ex} 
          & \multirow{2}{*}{Production} & \multirow{2}{*}{Consumption} & \multicolumn{1}{c}{Production} & \multicolumn{1}{c}{Production} & \multicolumn{1}{c}{Production} \\
        &   &   &  \multicolumn{1}{c}{Cost} &\multicolumn{1}{c}{Plan} & \multicolumn{1}{c}{Plan Difference} \\
    \hline
    Beijing  & 317.6 & 671.0  & 1.00  &  254.1  & -63.5 \\
    Tianjin  & 0.0   & 3348.7 & N/A   & 0.0   & 0.0  \\
    Hebei  & 6484.3  & 22249.4  & 0.19  & 5187.4  &  -1296.9 \\
    Shanxi  & 81641.5  & 28198.8  & 0.05  & 71403.7  &  -10237.8\\
    Inner Mongolia  & 83827.9  & 29033.4  & 0.02  & 67068.6  & -16759.3  \\
    Liaoning  & 4082.1  & 13413.2  & 0.04  & 6123.2 & 2041,1 \\
    Jilin  & 1643.1  & 7454.7  & 0.08  & 2464.6 & 821.5  \\
    Heilongjiang  & 5623.2  & 11110.1  & 0.04  & 8434.8  & 2811.6 \\
    Shanghai  & 0.0   & 3661.8  & N/A   & 0.0   & 0.0  \\
    Jiangsu  & 1367.9  & 22203.7  & 0.09  &   2051.8  & 683.9 \\
    Zhejiang  & 0.0   & 11042.1  & N/A   & 0.0   & 0.0  \\
    Anhui  & 12235.6  & 12451.4  & 0.08  & 18353.4  & 6117.8\\
    Fujian  & 1346.7  & 5404.1  & 0.06  & 2020.0 &  673.3\\
    Jiangxi & 1432.1  & 6030.3  & 0.06  & 2148.2  &  716.1 \\
    Shandong  & 12813.5  & 32408.9  & 0.14  & 17473.8 &  4660.3 \\
    Henan  & 11905.3  & 18386.9 & 0.16  & 9535.9  & -2369.4 \\
    Hubei  & 547.4  & 9250.9  & 0.11  &  821.1  & 273.7 \\
    Hunan  & 2595.5  & 9059.1  & 0.13  &  3893.2  &  1297.7\\
    Guangdong  & 0.0   & 12773.3  & N/A   & 0.0   & 0.0  \\
    Guangxi  & 399.6  & 5159.7  & 0.11  &  599.4 &  199.8 \\
    Hainan & 0.0   & 803.8  & N/A   & 0.0   & 0.0  \\
    Chongqing  & 2419.7  & 4492.0  & 0.07  & 3629.5  &  1209.8 \\
    Sichuan  & 6076.2  & 7021.4  & 0.13  & 7020.1  & 843.9 \\
    Guizhou  & 16662.2  & 10800.1  & 0.05  & 24993.3  & 8331.1  \\
    Yunnan  & 4251.8  & 5906.5  & 0.04  & 6377.7  & 2125.9  \\
    Shaanxi  & 51151.4  & 15572.0  & 0.02  & 47972.9 &  -3178.5 \\
    Gansu  & 4236.9  & 5048.7  & 0.03  & 5440.3  & 1203.4 \\
    Qinghai  & 774.6  & 1553.5 & 0.01  &   1161.9  &  387.3\\
    Ningxia  & 6728.4  & 6859.6  & 0.03  & 6859.6  &  131.2 \\
    Xinjiang  & 15834.0  & 15029.2 & 0.01  & 15029.2  & -804.8\\
    \hline
    \end{tabular}%
    \caption{The main production and consumption of coal in the Chinese mainland. Column 2 is the raw coal production of enterprises above the scale, and column 3 is the coal consumption. Column 4 is the relative cost to produce one unit of coal. 
    Column 5 is the optimal production plan computed by the OPT model. 
    Column 6 is the difference of the optimal production plan between OPT and the classical OT. }
  \label{tab:coal_basic}%
  \end{scriptsize}
\end{table}%

In \cref{tab:coal_basic}, we list the main production of raw coal and the consumption of coal in the Chinese mainland in 2016, and normalize the raw data to ensure that total production and consumption are equal\footnote{Data source: National Bureau of Statistics of China}. 
The production cost is set based on the average cost of coal among these provinces and municipalities. 
We use `N/A' (not available) to represent that the relative production cost with the production is zero. 
We consider simplifying the transport cost to be proportional to the spatial distance.
Thus, the transport cost of transporting one unit of coal is calculated by the longitude and latitude of provincial capitals and municipalities. The total cost is defined as the summation of the normalized production cost to the normalized transport cost. 
The production of coal cannot be arbitrarily large or arbitrarily small, due to the limitation of equipment and natural resources and the coal exploitation can not be shut down at will. Based on the real production curves in recent years, we set the upper bound of production as $1.5$ times the actual values  and the lower bound as $0.8$ times.

\begin{figure}[t] 
    \centering
    \includegraphics[scale=0.75, trim=100 50 100 50, clip]{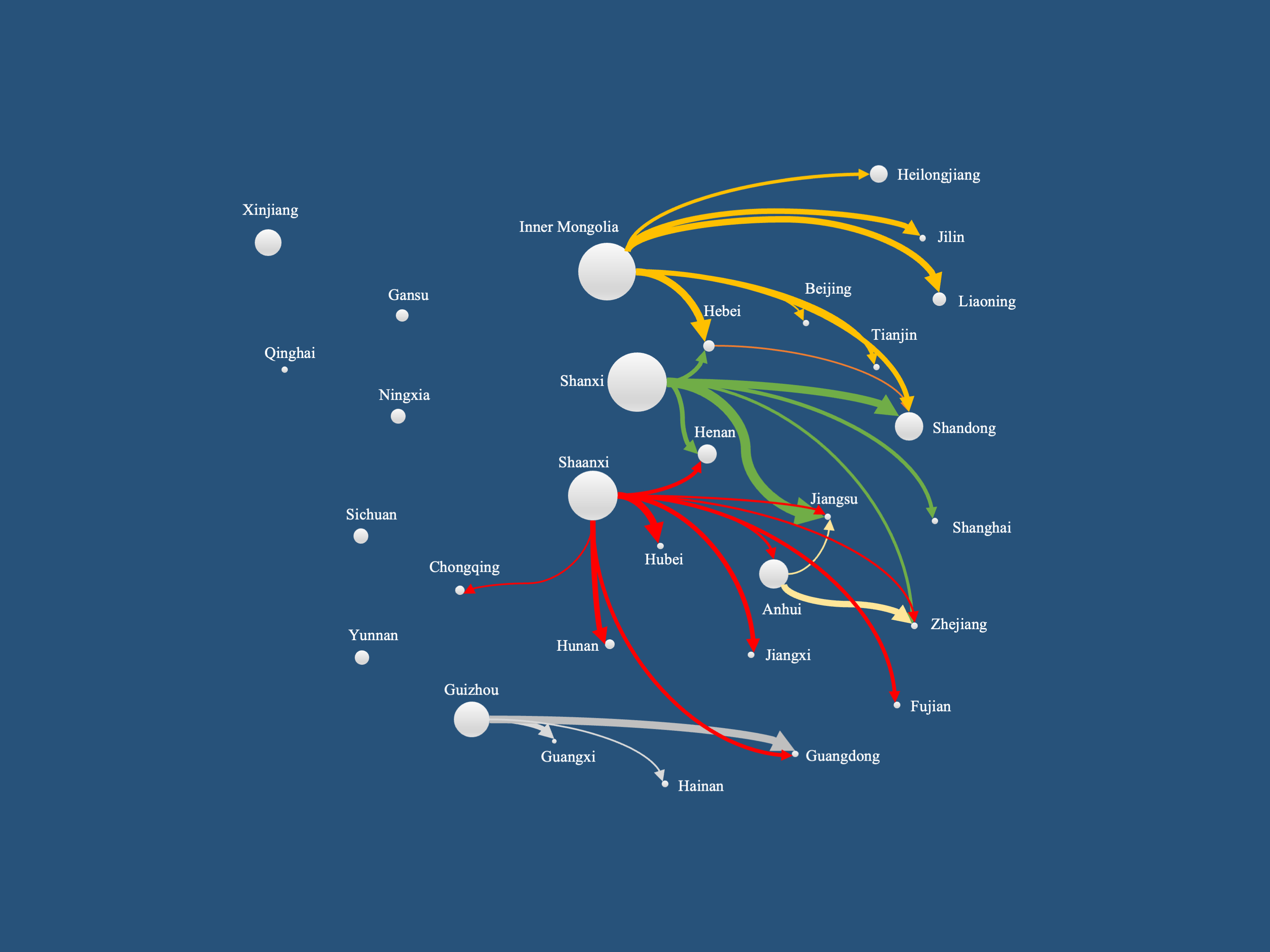}
    \caption{The main production plan and transport plan between different areas in the Chinese mainland.}
    \label{fig:coal}
\end{figure}

After modeling this practical problem to the OPT model, we utilize Gurobi and GAS-I to solve this model, in other words, to provide suggestions for the feasible production and transport of coal with minimum cost. The minimum cost computed by Gurobi and GAS-I is $26929.98$ and $27039.13$ respectively, and the relative error is $4.05\times 10^{-3}$. The minimum cost computed by the classical OT model is $31140.34$, and our OPT model can reduce the cost by $13.17\%$.
We show the production plan in Column 5 of \cref{tab:coal_basic} and the main transport plan between different areas in \cref{fig:coal}. We can see that the bulk of coal is exploited in Inner Mongolia, Shanxi, and Shaanxi, which are the main coal-producing provinces and produce more than half of coal according to the report from the National Bureau of Statistics of China. As for consumption, coal is mainly consumed in the northeastern industrial region and the southeastern coastal region, which is also consistent with the industry situation of the Chinese mainland. 
We also show the production plan in Column 5 of \cref{tab:coal_basic} and the main transport plan between different areas in \cref{fig:coal}.
The difference of the production between the classical OT model and the OPT model is presented in Column 6 of \cref{tab:coal_basic}. According to the results, it is suggested to reduce the coal production in Inner Mongolia, Shanxi, and Shaanxi, and increase the coal production in Guizhou, Anhui, and Shandong.
The above numerical results demonstrate the expressiveness of our proposed OPT model.

\section{Conclusions and Outlook}
\label{sec:conclusions}

Optimal transport has achieved significant success in economics, while we notice that variable production has not been taken into account in the economic application of optimal transport. Thus, in this paper, we propose the Optimal Production Transport~(OPT) model, which is an extension of the classical OT. In the OPT model, the production plan and transport plan need to be obtained simultaneously with the minimum production and transport cost. 
To solve this novel model, 
we propose two well-designed algorithms called GAS-I and GAS-II by introducing the regularized multiple relaxation OPT model with multiple relaxation variables and multiple regularization terms.
The convergence and the number of iterations of GAS-II are theoretically guaranteed, and comprehensive numerical experiments show that our proposed GAS-I and GAS-II algorithms have a remarkable advantage in accuracy, efficiency, and memory usage compared with IBP and Gurobi.
Furthermore, we demonstrate practical applications of the OPT model by considering the coal production and transport problem as an example. The results illustrate that utilizing the OPT model reduces the overall production and transport cost compared to the traditional OT model.

\bibliographystyle{siamplain}
\bibliography{main_references}

\newpage
\appendix
\section{IBP for OPT model} \label{sec:IBP}

Iterative Bregman projections~(IBP) algorithms are proposed in \cite{IBP} for solving several variants of the optimal transport model, e.g., Wasserstein barycenter problem \cite{agueh2011barycenters}, multi-marginal optimal transport \cite{pass2011uniqueness, pass2012local}, capacity-constrained optimal transport \cite{COT}, partial optimal transport \cite{POT}. The main idea of IBP is to split the constraints into an intersection of a few simple constraints, and project the solution onto each simple constraint set alternately. 
The entropic regularization of the OPT model \cref{OPT_d} can be solved by the idea of IBP.
Now we introduce the IBP algorithm, which is specifically designed to solve the entropic regularization of the OPT model \cref{OPT_d}, and to state the convergence result of the entropic regularization for the OPT model.

In the IBP algorithm, we define the entropic regularization model and split the constraints into an intersection of four hyperplane or half-space constraints. The solution is projected onto each of the four constraint sets alternately with the iterations of matrix operations. The convergence result of the entropic regularization w.r.t. $\varepsilon$ can also be proved. 
However, given that the number of subspaces becomes larger compared with the classical OT, and some of them are not affine, the computational time and convergence of the IBP for the OPT have been highly affected \cite{huang2021projection,su2019order}.

\subsection{Numerical Algorithm}
Similar to the entropic regularization for the classical OT, we define the  entropic regularized optimal production transport model as
\begin{equation}\label{OPT_ER}
    \begin{aligned}
		&W= \min_{\bm{\gamma}} \quad \sum_{i=1}^{N}\sum_{j=1}^{N} (C_{ij} \gamma_{ij}+\varepsilon \gamma_{ij}(\ln \gamma_{ij}-1))\\
  		&\mathrm{s.t.}\quad  \bm{\hat{u}}\leq \bm{\gamma}\cdot \bm{1_{N}} \leq \bm{\overline{u}}, \ \bm{\gamma^{T}}\cdot \bm{1_{N}}=\bm{v},\ 
		\gamma_{ij}\leq \eta_{ij} .
\end{aligned}
\end{equation}

The definition of the Kullback-Leibler(KL) divergence between $\bm{\gamma}\in \mathbb{R}_{++}^{N\times N}$ and $\bm{K}\in \mathbb{R}_{++}^{N\times N}$, where $\mathbb{R}_{++}$ is the set of strictly positive real numbers, is
\begin{equation}
    \mathrm{KL}(\bm{\gamma}|\bm{K})\overset{def}{=} \sum_{i,j=1}^{N} (\gamma_{ij} \ln (\frac{\gamma_{ij}}{Ki_{ij}})-\gamma_{ij}+K_{ij}).
\end{equation}
Given a convex set $\mathcal{C}\in\mathbb{R}^{N}\times \mathbb{R}^N$, the projection according to the KL divergence is defined as
\begin{equation}
   P_{\mathcal{C}}^{\mathrm{KL}}(\bm{K})\overset{def}{=}\mathrm{argmin}_{\bm{\gamma}\in \mathcal{C}} \mathrm{KL}(\bm{\gamma}|\bm{K}).
\end{equation}

So \cref{OPT_ER} can be recast in the form
\begin{align}
		W= \min_{\bm{\gamma}\in \mathcal{C}} \quad \mathrm{KL}(\bm{\gamma}|\bm{K}),
\end{align}
where $\bm{K}=e^{-\frac{\bm{C}}{\varepsilon}}$, $\mathcal{C}=\mathcal{C}_1\cap \mathcal{C}_2  \cap \mathcal{C}_3 \cap \mathcal{C}_4$, $\mathcal{C}_1=\{\bm{\gamma}\in \mathbb{R}^{N\times N}|\bm{\gamma}\cdot \bm{1_{N}}\geq \bm{\hat{u}}\}$, $\mathcal{C}_2=\{\bm{\gamma}\in \mathbb{R}^{N\times N}|\bm{\gamma}\cdot \bm{1_{N}}\leq \bm{\overline{u}}\} $, $\mathcal{C}_3=\{\bm{\gamma}\in \mathbb{R}^{N\times N}| \bm{\gamma^{T}}\cdot \bm{1_{N}} = \bm{v} \}$, $\mathcal{C}_4=\{\bm{\gamma}\in \mathbb{R}^{N\times N}| \bm{\gamma}\leq \bm{\eta} \}$.

\begin{proposition}
Denoting $	\bm{\gamma}^{(n)}=\mathrm{argmin}_{\bm{\gamma}\in C_n} \mathrm{KL}(\bm{\gamma}|\bm{\gamma}^{(n-1)})$, and $\mathcal{C}_{n+4} = \mathcal{C}_{n}$ for the sake of simplicity, one has
\begin{equation*}
    \begin{aligned}
    	& (a) \quad \mathrm{argmin}_{\bm{\gamma}\in \mathcal{C}_1} \mathrm{KL}(\bm{\gamma}|\bm{\gamma}^{(n-1)})=\mathrm{diag} \left(\max \left(\frac{\bm{\hat{u}}}{\bm{\gamma}^{(n-1)}\cdot \bm{1_{N}}},\bm{1}_N\right)\right) \bm{\gamma}^{(n-1)}.\\
     	& (b) \quad  \mathrm{argmin}_{\bm{\gamma}\in \mathcal{C}_2} \mathrm{KL}(\bm{\gamma}|\bm{\gamma}^{(n-1)})=\mathrm{diag} \left(\min \left(\frac{\bm{\overline{u}}}{\bm{\gamma}^{(n-1)}\cdot \bm{1_{N}}},\bm{1}_N\right)\right) \bm{\gamma}^{(n-1)}.\\
      	& (c) \quad  \mathrm{argmin}_{\bm{\gamma}\in \mathcal{C}_3} \mathrm{KL}(\bm{\gamma}|\bm{\gamma}^{(n-1)})=\bm{\gamma}^{(n-1)} \mathrm{diag} \left(\frac{\bm{v}}{(\bm{\gamma}^{(n-1)})^T \cdot \bm{1_{N}}}\right).\\
      	& (d) \quad  \mathrm{argmin}_{\bm{\gamma}\in \mathcal{C}_4} \mathrm{KL}(\bm{\gamma}|\bm{\gamma}^{(n-1)})=\min(\bm{\gamma}^{(n-1)},\bm{\eta}).
\end{aligned}
\end{equation*}

\end{proposition}
\begin{proof}
    The correctness of (b), (c), (d) can be found in \cite{IBP}, we just need to prove that (a) is correct.
\begin{equation}
        \frac{\partial \mathrm{KL}(\bm{\gamma}|\bm{\gamma}^{(n-1)})}{\partial \gamma_{ij}}=\ln \gamma_{ij}-\ln \gamma^{(n-1)}_{ij}\left\{
    \begin{aligned}
        <0,\quad \gamma_{ij}<\gamma^{(n-1)}_{ij},\\
        =0,\quad \gamma_{ij}=\gamma^{(n-1)}_{ij},\\
        >0,\quad \gamma_{ij}>\gamma^{(n-1)}_{ij}.
    \end{aligned}
    \right.
\end{equation}
So $\gamma_{ij}$ is closer to $\gamma^{(n-1)}_{ij}$, $\mathrm{KL}(\bm{\gamma}|\bm{\gamma}^{(n-1)})$ is smaller. 
Considering that $\bm{\gamma}\in \mathcal{C}_1= \{\bm{\gamma}\in \mathbb{R}^{N\times N}|\bm{\gamma}\cdot \bm{1_{N}}\geq \bm{\hat{u}}\}$, if $\sum_{j=1}^{N} \gamma^{(n-1)}_{ij}\geq \hat{u}_i$ for some $j$ holds, then for these $j$ and any $i$, the optimal solution has to have $\gamma_{ij}=\gamma^{(n-1)}_{ij}$, which is consistent with the conclusion (a). 
It is easy to find that for fixed $j$, if $\sum_{j=1}^{N} \gamma^{(n-1)}_{ij}< \hat{u}_i$, the optimal solution for $ \mathrm{KL}(\bm{\gamma}|\bm{\gamma}^{(n-1)})$ is to increase $\gamma^{(n-1)}_{ij}$ in equal proportions so that the condition $\sum_{j=1}^{N} \gamma_{ij}= \hat{u}_i$ is satisfied, which is also consistent with the conclusion (a). So the conclusions (a) holds.

Summarize the above, the conclusions (a-d) are correct.
\end{proof}

Now we introduce the IBP algorithm for the OPT model.
Starting from $\bm{\gamma}^{(0)}=\bm{K}$ and $\forall n \geq 1$ computing 
\begin{align*}
\bm{\gamma}^{(n)}=P_{\mathcal{C}_{n}}^{\mathrm{KL}}(\bm{\gamma}^{(n-1)})=
\left\{
\begin{aligned}
&\mathrm{diag} \left(\max \left(\frac{\bm{\hat{u}}}{\bm{\gamma}^{(n-1)}\cdot \bm{1_{N}}},\bm{1}_N\right)\right) \bm{\gamma}^{(n-1)},
&n \equiv 1\pmod{4},\\
&\mathrm{diag} \left(\min \left(\frac{\bm{\overline{u}}}{\bm{\gamma}^{(n-1)}\cdot \bm{1_{N}}},\bm{1}_N\right)\right) \bm{\gamma}^{(n-1)},
&n \equiv 2\pmod{4},\\
&\bm{\gamma}^{(n-1)} \mathrm{diag} \left(\frac{\bm{v}}{(\bm{\gamma}^{(n-1)})^T \cdot \bm{1_{N}}}\right),
&n \equiv 3\pmod{4},\\
&\min(\bm{\gamma}^{(n-1)},\bm{\eta}),
&n \equiv 0\pmod{4}.
\end{aligned}
\right.
\end{align*}

\subsection{Convergence Analysis}
Now we prove the convergence result of the entropic regularization w.r.t. $\varepsilon$.
\begin{theorem}
	The optimal solution of model \cref{OPT_ER} is unique. Let $\bm{\gamma}_{\varepsilon}$ be the optimal solution of model \cref{OPT_ER} with $\bm{u}_{\varepsilon}=\bm{\gamma}_{\varepsilon}\cdot \bm{1}_{N}$. There is ($\bm{\gamma}^*,\bm{ u}^*$), which is one of the optimal solutions of \cref{OPT_d}, so that $\bm{\gamma}_{\varepsilon}$ converges to $\bm{\gamma}^*$, $\bm{u}_{\varepsilon}$ converges to $\bm{u}^*$ as $\varepsilon\rightarrow 0$.
\end{theorem}
\begin{proof}
Firstly, we prove that the optimal solution of \cref{OPT_ER} is unique. As we all know, $h(x)=x(\ln x-1)$ is strictly convex functions, so the model \cref{OPT_ER} is still strictly convex and has a unique optimal solution.

 Then we prove the convergence. We can find that $\bm{u}^*=\bm{\gamma}^*\cdot \bm{1_{N}}$ and $\bm{u}_{\varepsilon}=\bm{\gamma}_{\varepsilon}\cdot \bm{1_{N}}$, so $\bm{u}_{\varepsilon}$ converges to $\bm{u}^*$ if and only if $\bm{\gamma}_{\varepsilon}$ converges to $\bm{\gamma}^*$,
 so we just need to prove that $\bm{\gamma}_{\varepsilon}$ converges to $\bm{\gamma}^*$.
 
 Now we prove that $\bm{\gamma}_{\varepsilon}$ converges to $\bm{\gamma}^*$ as $\varepsilon\rightarrow 0$.
 For any sequence $\{\varepsilon_k\}_{k=1}^{\infty}$, where $\varepsilon_k>0$ and $\lim_{k\rightarrow \infty} \varepsilon_k=0$. Suppose the optimal solution of \cref{OPT_ER} with the regularization parameter $\varepsilon_k$ is $\bm{\gamma}_k$.
 $\bm{\gamma}_k$ is a feasible solution of \cref{OPT_ER} if and only if 
 $$\bm{\gamma}_k \in \Gamma'=\{\bm{\hat{u}}\leq \bm{\gamma}\cdot \bm{1_{N}} \leq  \bm{\overline{u}}, \bm{\gamma}^{T}\cdot \bm{1_{N}}=\bm{v}, \bm{0}\leq \bm{\gamma} \leq \bm{\eta}\}.$$ 
 Consider that $\Gamma'$ is closed and bounded, there exists a subsequence of $\bm{\gamma}_k$ converging to $\hat{\bm{\gamma}}\in \Gamma'$. For the sake of simplicity, we still use the same symbol $\bm{\gamma}_k$ to represent the subsequence.
 
 Now we can find $\bm{\gamma}_{k}, \bm{\gamma}^*$ and $\hat{\bm{\gamma}}$ are all feasible solution of \cref{OPT_d} and \cref{OPT_ER}. Because $\bm{\gamma}^*$ is the optimal solution of \cref{OPT_d}, so for any $\bm{\gamma}_k$, $\langle \bm{C}, \bm{\gamma}^*\rangle \leq \langle \bm{C}, \bm{\gamma}_k \rangle$ holds. 
 Because $\bm{\gamma}_k$ is the optimal solution of \cref{OPT_ER} with $\varepsilon=\varepsilon_k$, so 
 \begin{equation*}
     \begin{aligned}
     \langle \bm{C}, \bm{\gamma}_k\rangle +\varepsilon_k(\langle \bm{\gamma}_k, \ln \bm{\gamma}_k \rangle \leq \langle \bm{C}, \bm{\gamma}^* \rangle +\varepsilon_k(\langle \bm{\gamma}^*, \ln \bm{\gamma}^* \rangle
 \end{aligned}
 \end{equation*}
 holds,
 then we have
 \begin{equation*}
      \begin{aligned}
     0\leq \langle \bm{C}, \bm{\gamma}_k\rangle - \langle \bm{C}, \bm{\gamma}^* \rangle
     \leq \varepsilon_k(\langle \bm{\gamma}^*, \ln \bm{\gamma}^* \rangle-\langle \bm{\gamma}_k, \ln \bm{\gamma}_k \rangle ).
 \end{aligned}
 \end{equation*}
 Note that the inner product function and logarithmic function are continuous and the definition domain is closed and bounded, so each term in the right-hand equation is bounded with different $k$. 
 Now we let the limit $k\rightarrow \infty$, then $\varepsilon_k\rightarrow 0$, and we can find that $\langle \bm{C}, \bm{\gamma}_k\rangle - \langle \bm{C}, \bm{\gamma}^* \rangle \leq 0$ holds, which means that $\langle \bm{C}, \bm{\gamma}_k\rangle - \langle \bm{C}, \bm{\gamma}^* \rangle = 0$ and $\hat{\bm{\gamma}}$ is the optimal solution of \cref{OPT_d}.
\end{proof}

\end{document}